\documentclass[preprint,12pt]{elsarticle}

\makeatletter
\def\ps@pprintTitle{%
 \let\@oddhead\@empty
 \let\@evenhead\@empty
 \def\@oddfoot{}%
 \let\@evenfoot\@oddfoot}
\makeatother

\usepackage{graphicx}
\usepackage{tikz}
\usetikzlibrary{decorations.pathmorphing}
\usetikzlibrary{arrows,shapes,matrix,decorations,shapes.geometric}

\usepackage{amssymb}
\usepackage{amsthm}

\newtheorem{thm}{Theorem}
\newtheorem{lem}[thm]{Lemma}
\newtheorem{prop}[thm]{Proposition}
\newtheorem{cor}[thm]{Corollary}
\newtheorem{obs}[thm]{Observation}
\newtheorem{conj}[thm]{Conjecture}

\newcommand{\st}{\colon\ }

\journal{Discrete Mathematics}

\begin{document}

\begin{frontmatter}


\author[FC]{Sebasti\'an Gonz\'alez Hermosillo de la Maza\corref{cor1}}
\ead{sghm@matem.unam.mx}

\author[FC]{C\'esar Hern\'andez-Cruz}
\ead{chc@ciencias.unam.mx}

\address[FC]{Facultad de Ciencias, UNAM\\
						Av. Universidad 3000, Circuito Exterior S/N\\
						Delegaci\'on Coyoac\'an, C.P. 04510\\
						Ciudad Universitaria, D.F., M\'exico}

\cortext[cor1]{Corresponding Author}

\title{On the existence of $3$- and $4$-kernels in digraphs}


\author{}

\address{}

\begin{abstract}
Let $D = (V(D), A(D))$ be a digraph.   A subset $S \subseteq V(D)$ is
$k$-{\em independent} if the distance between every pair of vertices of $S$ is at
least $k$, and it is $\ell$-{\em absorbent} if for every vertex $u$ in $V(D) \setminus S$
there exists $v \in S$ such that the distance from $u$ to $v$ is less than or equal to
$\ell$.   A $k$-{\em kernel} is a $k$-independent and $(k-1)$-absorbent set.   A
kernel is simply a $2$-kernel.

A classical result due to Duchet states that if every directed cycle in a digraph $D$ has at
least one symmetric arc, then $D$ has a kernel.   We propose a conjecture generalizing
this result for $k$-kernels and prove it true for $k = 3$ and $k = 4$.
\end{abstract}


\end{frontmatter}


\section{Introduction}
\label{sec:Intro}

Since their introduction by von Neumann and Morgenstern in \cite{VNM} in
the context of winning strategies in game theory, digraph kernels have been
widely studied in different contexts.   Kernels in digraphs gained a lot of
attention for their relation with the Strong Perfect Graph Conjecture (now the
Strong Perfect Graph Theorem), but their applications ranges from list
edge-colourings of graphs to Model Theory in Mathematical Logic \cite{BGW}.
Chv\'atal proved in \cite{C} that the problem of determining whether a digraph
has a kernel (the {\em kernel problem}) is $\mathcal{NP}$-complete.

Kwa\'snik and Borowiecki introduced in \cite{K1} the concept of $(k,\ell)$-kernel of 
a digraph $D$ as a $k$-independent and $\ell$-absorbing subset of $V(D)$.   It is
easily observed that arbitrary choices of $k$ and $\ell$ do not usually lead to problems
similar to the kernel problem.   As an example, consider the following basic result for
kernels.   Every acyclic digraph has a unique kernel, and this kernel can be recursively
constructed.   But, for $k > \ell - 1$, directed paths of length greater than $\ell$ do not
have a $(k,\ell)$-kernel.   Observing that a kernel is a $(2, 1)$-kernel, it is natural to
consider the special case of $(k, k-1)$-kernels, which are known as $k$-kernels, as
they have similar properties to those of the usual kernels.   In \cite{HH}, Hell and
Hern\'andez-Cruz proved that determining whether a cyclically $3$-partite digraph
with circumference $6$ has a $3$-kernel is an $\mathcal{NP}$-complete problem,
and hence, the kernel problem remains $\mathcal{NP}$-complete even for $3$-colorable
digraphs.

Finding sufficient conditions for the existence of $k$-kernels has been a fruitful line
of work.   For the $2$-kernel case there are a lot of easy to verify and elegant such conditions,
for example, acyclic digraphs, transitive digraphs, symmetric digraphs, bipartite digraphs,
and digraphs without odd directed cycles have a kernel.   A particularly nice and useful
theorem was proved by Duchet in 1980, \cite{Duchet}.

\begin{thm}\label{Duchet}
If every directed cycle in $D$ has at least one symmetric arc, then $D$ is kernel-perfect.
\end{thm}

Observe that not all of these conditions have an analogue for $k$-kernels when $k \ge 3$.
Knowing that bipartite digraphs have a kernel, it seems natural to ask whether $3$-colorable
digraphs have a kernel.   Sadly, this is not true; moreover, as mentioned above, it is
$\mathcal{NP}$-complete to determine whether a cyclically $3$-partite digraph has a
$3$-kernel, \cite{HH}.   The first interesting generalization to $k$-kernels of a known result
for kernels is due to Kwa\'snik, who proved in 1981, \cite{K2}, that if all directed cycles of a
strongly connected digraph $D$ have length $\equiv 0$ (mod $k$), then $D$ has a
$k$-kernel.

For the last 34 years, the study of sufficient conditions for the existence of $k$-kernels has
been focused on certain well-behaved families of digraphs.   The following families of
digraphs have been either proved to have a $k$-kernel for every integer $k \ge 3$, or their
members having a $k$-kernel have been characterized:   Multipartite tournaments
\cite{GS-HC1}, $k$-transitive digraphs \cite{HC-GS}, $k$-quasi-transitive digraphs
\cite{HC-GS, Wang}.   So, a general, non structure-dependent sufficient condition for the
existence of $k$-kernels is still missing for every integer $k \ge 3$.

For undefined terms, we refer the reader to \cite{BM}.  All digraphs considered here are
finite, without loops and without parallel arcs in the same directon.   A subset $S \subseteq
V(D)$ is $k$-{\em independent} if the distance between every pair of vertices of $S$ is at
least $k$, and it is $\ell$-{\em absorbent} if for every vertex $u$ in $V(D) \setminus S$
there exists $v \in S$ such that the distance from $u$ to $v$ is less than or equal to
$\ell$.   A $(k,\ell)$-{\em kernel} is a $k$-independent and $k$-absorbent set.   A
$k$-{\em kernel} is a $(k,k-1)$-kernel and a kernel is simply a $2$-kernel.

If $W = (x_0 \dots, x_n)$ is a (directed) walk in a digraph $D$, for $i < j$, $x_i W x_j$
will denote the walk $(x_i, x_{i+1}, \dots, x_{j-1}, x_j)$.   Union of walks will be denoted
by concatenation or with $\cup$.   The length of the walk $W$ will be denoted by $\left\|
W \right\|$.   All paths and cycles will be considered to be directed unless otherwise
stated.   A $k$-cycle is a cycle of length $k$.   In particular $3$-cycles will be often
called simply triangles.

If $D$ is a digraph, the $k$-closure of $D$, $\mathcal{C}^k (D)$ is the digraph with
vertex set $V(\mathcal{C}^k (D)) = V(D)$ and such that $(u,v) \in A(\mathcal{C}^k (D))$ if
and only if $d_D (u,v) \le k$.

We propose the following conjecture to generalize Theorem \ref{Duchet} to $k$-kernels.

\begin{conj} \label{Conj-k}
If every directed cycle $B$ in a digraph $D$ has at least $ \frac{k-2}{k-1} \left\| B \right\|  + 1$
symmetric arcs, then $D$ has a $k$-kernel.
\end{conj}

Clearly, the case $k = 2$ is precisely the statement in Theorem \ref{Duchet}.   The main
contribution of this work is proving Conjecture \ref{Conj-k} true for $k = 3$ and $k = 4$.

The rest of the paper is organized as follows.   In Section \ref{sec:Prelim} some technical
results are proved, and a family of digraphs showing that, if true, Conjecture \ref{Conj-k}
is sharp, is constructed.   Conjecture \ref{Conj-k} is proved for $k = 3$ and $k = 4$ in
Sections \ref{duc3} and \ref{duc4}, respectively.

\section{Generalizing Duchet for $k$-kernels}
\label{sec:Prelim}

Our arguments rely heavily on Theorem \ref{Duchet} and the following result relating the existence of kernels in the $(k-1)$-closure of a digraph $D$ with the existence of $k$-kernels in $D$, \cite{Cesar2}.

\begin{lem} \label{k-clo}
Let $k \ge 2$ be an integer.   The digraph $D$ has a $k$-kernel if and only if its $(k-1)$-closure, $\mathcal{C}^{k-1} (D)$, has a kernel.
\end{lem}

The basic idea of the proofs of Theorem \ref{3k} and Theorem \ref{4k}, concerning the generalization of Duchet's result for $3$ and $4$-kernels respectively, is to show that every cycle of the corresponding closure has a symmetric arc. In order to check that, we must exhaustively verify the ways in which a cycle in the closure can arise. 

We start by stating a result about the general case. Lemma \ref{disjoint} deals with a particular configuration of arcs that originates a cycle in the $(k-1)$-closure of a digraph $D$. Despite the fact that it only solves a very special configuration, it is valid for every integer $k \geq 2$.

\begin{lem}\label{disjoint}
Let $D$ be a digraph and $k$ a positive integer such that every cycle $C$ of $D$ has at least $\frac{k-2}{k-1} \left\| C \right\|  +1$ symmetric arcs and $B$ be a cycle of $ H = \mathcal{C}^{k-1} (D)$. If the paths of $D$ that give rise to the arcs of $A(C) \setminus A(D)$ are mutually internally disjoint, then $B$ has a symmetric arc.

\end{lem}

\begin{proof}

Let $B = (v_1, v_2, \dots, v_n, v_1)$ be a cycle of $\mathcal{C}^{k-1} (D)$ that satisfies the hypothesis. For $1 \leq i \leq n-1$, let $T_i$ be the $v_iv_{i+1}$-path in $D$ that originates the arc $(v_i,v_{i+1})$ in $H$. Similarly, we will use $T_n$ to denote the $v_nv_1$-path in $D$ that gives rise to the arc $(v_n,v_1)$. Let $\mathcal{T} = \left\{T_i \st 1 \le i \le n\right\}$. Also, for every $j \in \left\{2,3,\dots, k-1\right\}$, $m_j$ will denote the number elements of length $j$ in $\mathcal{T}$.

It is clear that by joining the paths in  $\mathcal{T}$ with the arcs in $A(C) \cap A(D)$ in the natural way we get a cycle of length $$n - \left(\sum^{k-1}_{j = 2}m_j\right) + \sum^{k-1}_{j = 2} jm_j = n + \sum^{k-1}_{j = 2} (j-1)m_j.$$

Also, since the length of $C$ is $n$, we have that $$m_{k-1} \leq n, m_{k-1} + m_{k-2} \leq n, \dots, \sum^{k-1}_{j = 2} m_j \leq n,$$ so the addition of these inequalities yields $$ \sum^{k-1}_{j = 2} (j-1)m_j \leq (k-2)n.$$

By adding $\sum^{k-1}_{j = 2} (k-2)(j-1)m_j$ on both sides of the inequality, we have

\begin{center}
\begin{eqnarray*}
  \sum^{k-1}_{j = 2} (j-1)m_j + \sum^{k-1}_{j = 2} (k-2)(j-1)m_j &\leq& (k-2)n + \sum^{k-1}_{j = 2} (k-2)(j-1)m_j 
\end{eqnarray*}
\end{center}

Performing algebraic manipulations on both sides we get the following

\begin{center}
\begin{eqnarray*}
\sum^{k-1}_{j = 2} (k-1)(j-1)m_j &\leq& (k-2)\left[n + \sum^{k-1}_{j = 2} (j-1)m_j \right]\\
	(k-1)\sum^{k-1}_{j = 2} (j-1)m_j &\leq& (k-2)\left[n + \sum^{k-1}_{j = 2} (j-1)m_j \right]\\
		\sum^{k-1}_{j = 2} (j-1)m_j &\leq& \frac{(k-2)}{(k-1)} \left[n + \sum^{k-1}_{j = 2} (j-1)m_j \right]
\end{eqnarray*}
\end{center}

The last inequality and the hypothesis about the number of symmetric arcs in the cycles of $D$ imply that at least one arc of $B$ is symmetric.

\end{proof}

The following example shows that, if true, the bound proposed by Conjecture \ref{Conj-k} is sharp. Let $k$ be an integer, with $k \geq 3$ and $V_k$, $U_k$ and $W_k$ be disjoint sets with $k-1$ elements. We will use $v_i$, $u_i$ and $w_i$ to denote its elements, respectively, for $1 \leq i \leq k-1$. Let $E = \left\{e_v,f_v,e_u,f_u,e_w,f_w\right\}$ be a set disjoint of $V_k$, $U_k$ and $W_k$.

Let \index{Digraphs!$H_k$}$H_k$ be the digraph such that its vertex set is $V(H_k) = V_k \cup U_k \cup W_k \cup E$ and its arc set is formed by:

\begin{itemize}

\item The arcs $(v_i,v_{i+1})$, $(u_i,u_{i+1})$ and $(w_i,w_{i+1})$, for every $1 \leq i \leq k-2$.

The arcs $(v_i,v_{i-1})$, $(u_i,u_{i-1})$ and $(w_i,w_{i-1})$, for every $2 \leq i \leq k-1$.

\item The arcs $(v_{k-1},u_1)$, $(u_{k-1},w_1)$ and $(w_{k-1},v_1)$.

\item The arcs $(s_{k-1},e_s)$ and $(e_s,f_s)$, for every $s \in \left\{v,u,w\right\}$.

\end{itemize}

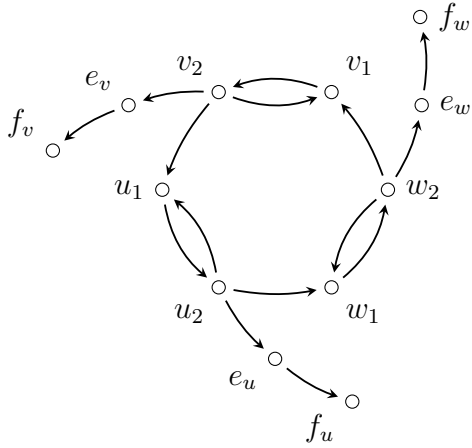
\begin{figure}[ht]
\begin{center}
\begin{tikzpicture}[scale=1]

\node (1) at (60:1.5)[circle, draw,inner sep=0pt, minimum width=5pt,label=60:$v_1$]{};
\node (2) at (120:1.5)[circle, draw,inner sep=0pt, minimum width=5pt,label=120:$v_2$]{};
\node (3) at (180:1.5)[circle, draw,inner sep=0pt, minimum width=5pt,label=180:$u_1$]{};
\node (4) at (240:1.5)[circle, draw,inner sep=0pt, minimum width=5pt,label=240:$u_2$]{};
\node (5) at (300:1.5)[circle, draw,inner sep=0pt, minimum width=5pt,label=300:$w_1$]{};
\node (6) at (0:1.5)[circle, draw,inner sep=0pt, minimum width=5pt,label=0:$w_2$]{};

\node (2a) at (150:2.25)[circle, draw,inner sep=0pt, minimum width=5pt,label=140:$e_v$]{};
\node (2b) at (170:3)[circle, draw,inner sep=0pt, minimum width=5pt,label=160:$f_v$]{};

\node (4a) at (270:2.25)[circle, draw,inner sep=0pt, minimum width=5pt,label=190:$e_u$]{};
\node (4b) at (290:3)[circle, draw,inner sep=0pt, minimum width=5pt,label=190:$f_u$]{};

\node (6a) at (30:2.25)[circle, draw,inner sep=0pt, minimum width=5pt,label=0:$e_w$]{};
\node (6b) at (50:3)[circle, draw,inner sep=0pt, minimum width=5pt,label=0:$f_w$]{};

\foreach \from/\to in {2/3,4/5,6/1,2/2a,2a/2b,4/4a,4a/4b,6/6a,6a/6b}
\draw [->, shorten <=3pt, shorten >=3pt, >=stealth, line width=.7pt] (\from) to [bend right = 10] (\to);

\foreach \from/\to in {1/2,2/1,3/4,4/3,5/6,6/5}
\draw [->, shorten <=3pt, shorten >=3pt, >=stealth, line width=.7pt] (\from) to [bend right = 20] (\to);

\end{tikzpicture}
\end{center}
\caption{The digraph $H_3$.}\label{H3}
\end{figure}

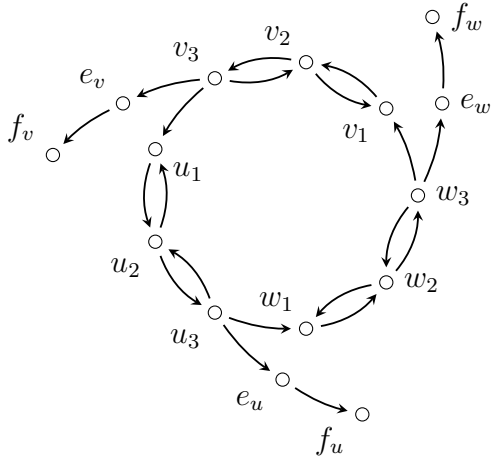
\begin{figure}[ht]
\begin{center}
\begin{tikzpicture}[scale=1]

\node (1) at (40:1.8)[circle, draw,inner sep=0pt, minimum width=5pt,label=220:$v_1$]{};
\node (2) at (80:1.8)[circle, draw,inner sep=0pt, minimum width=5pt,label=120:$v_2$]{};
\node (3) at (120:1.8)[circle, draw,inner sep=0pt, minimum width=5pt,label=120:$v_3$]{};

\node (4) at (160:1.8)[circle, draw,inner sep=0pt, minimum width=5pt,label=340:$u_1$]{};
\node (5) at (200:1.8)[circle, draw,inner sep=0pt, minimum width=5pt,label=240:$u_2$]{};
\node (6) at (240:1.8)[circle, draw,inner sep=0pt, minimum width=5pt,label=240:$u_3$]{};

\node (7) at (280:1.8)[circle, draw,inner sep=0pt, minimum width=5pt,label=100:$w_1$]{};
\node (8) at (320:1.8)[circle, draw,inner sep=0pt, minimum width=5pt,label=0:$w_2$]{};
\node (9) at (0:1.8)[circle, draw,inner sep=0pt, minimum width=5pt,label=0:$w_3$]{};

\node (3a) at (150:2.45)[circle, draw,inner sep=0pt, minimum width=5pt,label=140:$e_v$]{};
\node (3b) at (170:3.1)[circle, draw,inner sep=0pt, minimum width=5pt,label=160:$f_v$]{};

\node (6a) at (270:2.45)[circle, draw,inner sep=0pt, minimum width=5pt,label=190:$e_u$]{};
\node (6b) at (290:3.1)[circle, draw,inner sep=0pt, minimum width=5pt,label=190:$f_u$]{};

\node (9a) at (30:2.45)[circle, draw,inner sep=0pt, minimum width=5pt,label=0:$e_w$]{};
\node (9b) at (50:3.1)[circle, draw,inner sep=0pt, minimum width=5pt,label=0:$f_w$]{};

\foreach \from/\to in {3/4,6/7,9/1,3/3a,3a/3b,6/6a,6a/6b,9/9a,9a/9b}
\draw [->, shorten <=3pt, shorten >=3pt, >=stealth, line width=.7pt] (\from) to [bend right = 10] (\to);

\foreach \from/\to in {1/2,2/1,2/3,3/2,4/5,5/4,5/6,6/5,7/8,8/7,8/9,9/8}
\draw [->, shorten <=3pt, shorten >=3pt, >=stealth, line width=.7pt] (\from) to [bend right = 20] (\to);

\end{tikzpicture}
\end{center}
\caption{The digraph $H_4$.}\label{H4}
\end{figure}

Notice that the only cycle in $H_k$ of length greater than two is $C = (v_1,v_2, \dots, v_{k-1}, u_1, u_2, \dots,u_{k-1}, w_1, w_2, \dots, w_{k-1}, v_1)$ and has length $3(k-1)$ and has exactly $3(k-2) = \frac{k-2}{k-1} (3(k-1)) = \frac{k-2}{k-1} \left\| C \right\|$ symmetric arcs. Nevertheless, $H_k$ has no $k$-kernel.

\begin{prop}
The digraph $H_k$ has no $k$-kernel.
\end{prop}

\begin{proof}

Suppose that $K$ is a $k$-kernel of $H_k$. Clearly, $\left\{f_v, f_u, f_w\right\} \subseteq K$, since they are the sinks of $H_k$. The sinks of $H_k$ clearly $(k-1)$-absorb every vertex in $V(H_k)$ except $v_1$, $u_1$ and $w_1$, so at least one of them must be included in $K$. Thanks to the symmetry of $H_k$, we can assume that $v_1 \in K$. Since $d(w_1,v_1) = k-1$, we have that $w_1$ is $(k-1)$-absorbed by $v_1$. However, $d(v_1,u_1) = k-1$, which means the vertex $u_1$ cannot be included in $K$, but neither is it $(k-1)$-absorbed by a vertex in $K$, contradicting the fact that $K$ is a $k$-kernel of $H_k$.

\end{proof}

This shows that a digraph $D$ such that every cycle of $D$ has at least $\frac{k-2}{k-1} \left\| C \right\|$ symmetric are does not necessarily have a $k$-kernel, which shows that we cannot drop the $+ 1$ in the hypothesis.

\section{$3$-kernels} \label{duc3}

Here we present a generalization of Duchet for $3$-kernels. The idea is to prove that the $2$-closure of a digraph whose cycles have a at least certain proportion of symmetric arcs satisfies the hypothesis of Theorem \ref{Duchet}. The important part is to notice that a cycle of the $2$-closure may not be a cycle of $D$, but the fact that there is a cycle in the $\mathcal{C}^2(D)$ gives us some information about the structure of $D$. 

We will now prove a generalization of Theorem \ref{Duchet} for $3$-kernels.

\begin{lem} \label{C_3}
Let $D$ be a digraph.   If every directed cycle $B$ in $D$ has at least $ \frac{1}{2} \left\| B \right\|  + 1$ symmetric arcs then, every $C_3$ in $H =\mathcal{C}^2 (D)$ has at least one symmetric arc.
\end{lem}

\begin{proof}

Let $C = (v_0, v_1, v_2,v_0)$ be a $3$-cycle of $H$. If every arc of $C$ is also an arc of $D$, then $C$ is symmetric. If only two of the arcs in $C$ are arcs of $D$, then we can assume without loss of generality that $(v_0,v_1), (v_1,v_2) \in A(D)$. Since $(v_2,v_0) \in A(H)$, there exists $w \in V(D)$ such that $(v_2,w)$ and $(w,v_0)$ are arcs of $D$. If $w = v_1$, we have that $(v_1,v_2)$ is a symmetric arc in $D$ and, therefore, in $H$. On the other hand, if $w \neq v_1$, then $(w,v_0,v_1,v_2,w)$ is a $4$-cycle in $D$, and from the main hypothesis we derive that $(v_0v_1)$, $(v_1,v_2)$ or both are symmetric in $H$.

Suppose now that only $(v_0,v_1)$ is an arc of $D$. Then there are vertices $v,w \in V(D)$ such that $(v_1,v), (v,v_2),(v_2,w)$ and $(w,v_0)$ are arcs of $D$. If $v = w$, then $(v_0,v_1,v,v_0)$ is a triangle of $D$ and the arc $(v_0,v_1)$ is symmetric in $H$. If $v \neq w$, then $(v_0,v_1,v,v_2,w,v_0)$ is a $5$-cycle in $D$, which has at least four symmetric arcs. If one of those is $(v_0,v_1)$, we are done. Otherwise, the pairs $(v_1,v), (v,v_2)$ and $(v_2,w),(w,v_0)$ are symmetric. The former case implies that $(v_1,v_2)$ is symmetric and the latter that $(v_2,v_0)$ is symmetric.

Finally, let us consider the case where none of the arcs in $C$ is an arc of $D$. Let $u,v,w$ be vertices of $D$ such that $(v_0,u),(u,v_1),(v_1,v), (v,v_2),(v_2,w)$ and $(w,v_0)$ are arcs of $D$. If $u = v_2$, it is easy to see that the both $(v_1,v_2)$ and $(v_2,v_0)$ are symmetric arcs in $H$. The cases $v = v_0$ and $w = v_1$ are analogous.

Assume that the vertices $u,v$ and $w$ are different from $v_0,v_1$ and $v_2$. If $u = v = w$, then every arc in $C$ is symmetric. If $u = v$ but $u \neq w$, then $(v_0,v,v_2,w,v_0)$ is a $4$-cycle and as such it has at least three symmetric arcs, implying that $(v_0,v)$, $(v,v_2)$ or both are symmetric arcs in $H$. The cases $ v \neq u = w$ and $v = w \neq u$ are analogous. If $u,v$ and $w$ are all different, then Lemma \ref{disjoint} guarantees the existence of a symmetric arc in $C$.

\end{proof}

\begin{thm}\label{3k}
Let $D$ be a digraph. If every directed cycle $B$ in $D$ has at least $ \frac{1}{2} \left\| B \right\|  + 1$ symmetric arcs, then every cycle in $H = \mathcal{C}^2(D)$ has a symmetric arc.
\end{thm}

\begin{proof}

Let $C$ be a cycle in $H$. We proceed by induction on the length of $C$. The case when $C$ has length three is covered by Lemma \ref{C_3}.

Suppose then that $C = (v_1,v_2, \dots, v_n,v_0)$ is an $n$-cycle in $H$. For every arc $(v_i, v_{i+1}) \in A(C) \setminus A(D)$ there is a vertex $v_{i(i+1)} \in V(D)$ such that $(v_i, v_{ij}, v_j)$ is a directed path in $D$.   If $v_{i(i+1)} \neq v_{j(j+1)}$ for every $i \neq j$, and $v_{i(i+1)} \neq v_j$ for every $1 \leq i, j \leq n$, then Lemma \ref{disjoint} gives us the existence of a symmetric arc of $C$.

Thus, we can assume that $v_{i(i+1)} = v_j$ for some $1 \leq i,j \leq n$ or that $v_{i(i+1)} = v_{j(j+1)}$ for some $i \neq j$. If $v_{i(i+1)} = v_{(i+1)(i+2)}$ for some $1 \le i \le n$, then $(v_i, v_{i(i+1)}, v_{i+2})$ is a path of length two in $D$, implying that $(v_i,v_{i+2}) \in A(H)$.   We can assume without loss of generality that $i = 1$, hence $C' = (v_1, v_3) \cup v_3 C v_1$ is a cycle of length $n-1$ in $H$, just as depicted in Figure \ref{Cf1}.   The induction hypothesis implies that $C'$ has at least one symmetric arc.   If $C'$ has a symmetric arc other than $(v_1, v_3)$, then $C$ has a symmetric arc.   Thus, we can assume that $(v_1, v_3)$ is a symmetric arc in $H$. We have two possibilities. If $(v_3, v_1) \in A(D)$, then $(v_1,v_{12},v_3,v_1)$ is a triangle in $D$. The main hypothesis guarantees that $(v_1,v_{12})$ or $(v_{12},v_3)$ is symmetric, implying that $(v_1,v_2)$ or $(v_2,v_3)$ is symmetric.  If $(v_3, v_1) \notin A(D)$, there is a vertex $w \in V(D)$ such that $(v_3,w,v_1)$ is a path in $D$. If $w = v_2$ or $w = v_{12}$, then $(v_1,v_2)$ and $(v_2,v_3)$ are symmetric. Otherwise, $(v_1,v_{12},v_3,w,v_1)$ is a $4$-cycle in $D$ and has at least three symmetric arcs. Hence, $(v_1,v_{12})$ or $(v_{12},v_3)$ is symmetric and, therefore, $(v_1,v_2)$ or $(v_2,v_3)$ is a symmetric arc of $C$.

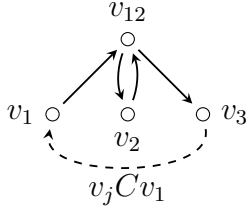
\begin{figure}[ht]
\begin{center}
\begin{tikzpicture}


\node (1) at (-1,0)[circle, draw,inner sep=0pt, minimum width=5pt,label=180:$v_1$]{};
\node (2) at (0,0)[circle, draw,inner sep=0pt, minimum width=5pt,label=270:$v_2$]{};
\node (3) at (1,0)[circle, draw,inner sep=0pt, minimum width=5pt,label=0:$v_3$]{};
\node (12) at (0,1)[circle, draw,inner sep=0pt, minimum width=5pt,label=90:$v_{12}$]{};
\node (C) at (0,-1)[]{$v_jCv_1$};

\foreach \from/\to in {1/12,12/3}
\draw [->, shorten <=3pt, shorten >=3pt, >=stealth, line width=.7pt] (\from) to [bend right = 0] (\to);

\foreach \from/\to in {2/12,12/2}
\draw [->, shorten <=3pt, shorten >=3pt, >=stealth, line width=.7pt] (\from) to [bend right = 20] (\to);

\foreach \from/\to in {3/1}
\draw [->, dashed, shorten <=3pt, shorten >=3pt, >=stealth, line width=.7pt] (\from) to [bend left = 100] (\to);

\end{tikzpicture}
\end{center}
\caption{The case $v_{12} = v_{23}$.}\label{Cf1}
\end{figure}

If $v_{i(i+1)} = v_j$ for some $1 \le i,j \le n$, then we can suppose that $j \notin \{ i-1, i+2 \}$, otherwise $(v_{i-1}, v_{i})$ or $(v_{i+1}, v_{i+2})$ would be a symmetric arc of $C$. Now, let $1 \le i \ne j \le n$ be such that $v_{i(i+1)} = v_{j(j+1)}$ or $v_{i(i+1)} = v_j$ and $|j-i|$ is minimum with this property.   We can assume without loss of generality that $i = 1$.

If $v_{12} = v_{j(j+1)}$ (see Figure \ref{Cf2}), then we have already observed that $j = 2$ implies the existence of a symmetric arc in $H$, so $j \ge 3$.   Let $P$ be the walk obtained from $v_2 C v_j$ by replacing every arc $(v_i, v_{i(i+1)}) \in A(C) \setminus A(D)$ with the path $(v_i, v_{i(i+1)}, v_{i+1})$ of $D$. Since we chose $|j-i|$ minimum, $P$ is a path in $D$ and $C' = P \cup (v_j, v_{12}, v_2)$ is a cycle in $D$ of length $k+j$, where $k = |A(v_2 C v_j) \setminus A(D)|$.  From the main hypothesis we derive that $C'$ has at least $ \frac{k+j}{2}  + 1$ symmetric arcs in $D$.  If there  is an arc $(v_i, v_{i+1}) \in A(P) \cap A(C)$ that is symmetric in $D$, then we have found a symmetric arc of $C$.   Otherwise, we have $ \frac{k+j}{2}  + 1$ symmetric arcs in the remaining $k+1$ pairs of arcs of $C'$.   But, $k \le j-2$ and hence $k \le  \frac{k+j-2}{2} $.   We obtain that $k+1 \le  \frac{k+j}{2} $.   Hence, the Pidgeonhole Principle implies that either a pair of arcs $(v_i, v_{i(i+1)})$ and $(v_{i(i+1)}, v_{(i+1)})$ of $C'$ are symmetric in $D$ or the arcs $(v_j, v_{12})$ and $(v_{12}, v_2)$ are symmetric in $D$.   In the former case, the arc $(v_i, v_{i+1})$ is a symmetric arc of $C$.

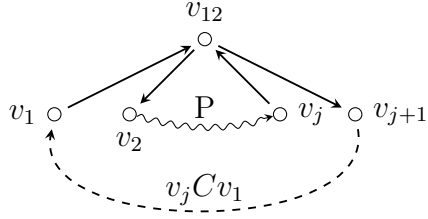
\begin{figure}[ht]
\begin{center}
\begin{tikzpicture}


\node (1) at (-1,0)[circle, draw,inner sep=0pt, minimum width=5pt,label=180:$v_1$]{};
\node (2) at (0,0)[circle, draw,inner sep=0pt, minimum width=5pt,label=270:$v_2$]{};
\node (3) at (2,0)[circle, draw,inner sep=0pt, minimum width=5pt,label=0:$v_{j}$]{};
\node (4) at (3,0)[circle, draw,inner sep=0pt, minimum width=5pt,label=0:$v_{j+1}$]{};

\node (12) at (1,1)[circle, draw,inner sep=0pt, minimum width=5pt,label=90:$v_{12}$]{};
\node (C) at (1,-1)[]{$v_jCv_1$};
\node (P) at (1,.1)[]{P};

\foreach \from/\to in {1/12,12/2,3/12,12/4}
\draw [->, shorten <=3pt, shorten >=3pt, >=stealth, line width=.7pt] (\from) to [bend right = 0] (\to);

\foreach \from/\to in {2/3}
\draw [->,decorate, decoration={snake,amplitude=.4mm,segment length=2mm,post length=1mm}, shorten <=0pt, shorten >=0pt, >=stealth, line width=.4pt] (\from) to [bend right = 15] (\to);

\foreach \from/\to in {4/1}
\draw [->,dashed, shorten <=3pt, shorten >=3pt, >=stealth, line width=.7pt] (\from) to [bend left = 100] (\to);

\end{tikzpicture}
\end{center}
\caption{The case $v_{12} = v_{j(j+1)}$.}\label{Cf2}
\end{figure}

In the latter case, let us observe that $(v_1, v_{12}, v_{j+1})$ is a path in $D$, and hence $C'' = (v_1, v_{j+1}) \cup v_{j+1} C v_1$ is a directed cycle in $H$ of length less than $n$.   Thus, we can derive from the induction hypothesis that  $C''$ has at least one symmetric arc.   If such symmetric arc is not $(v_1, v_{j+1})$, then we have already found a symmetric arc of $C$.   So, $(v_{j+1}, v_1) \in A(H)$, and we have two cases.   If $(v_{j+1}, v_1) \in A(D)$, then $(v_1, v_{12}, v_{j+1}, v_1)$ is a cycle of $D$.   Hence, the arcs $(v_1, v_{12})$ and $(v_{12}, v_{j+1})$ are symmetric in $D$.   We can conclude that $(v_2, v_1) \in A(H)$ and $(v_{j+1}, v_j) \in A(H)$.   Thus, we may assume that there is a vertex $x \in V(D)$ such that $(v_{j+1}, x, v_1)$ is a path in $D$.   If $x = v_{12}$, then $(v_2, v_{12}, v_1)$ is a path in $D$, and $(v_2, v_1)$ is a symmetric arc of $H$.   If $x \ne v_{12}$, then $(v_1, v_{12}, v_{j+1}, x, v_1)$ is a cycle in $D$.   Again, at least one of the arcs $(v_1, v_{12})$ or $(v_{12}, v_{j+1})$ is symmetric in $D$.   This implies that $(v_2, v_1) \in A(H)$ or $(v_{j+1}, v_j) \in A(H)$, as desired.

If $v_{12} = v_j$ (see Figure \ref{Cf3}), then we have already observed that $j \in \{ n, 3 \}$ implies the existence of a symmetric arc of $C$.   Thus, since $D$ is loopless, we can consider $j \notin \{ 1, 2, 3, n \}$.   By an argument similar to the previous case, we obtain the path $P$ replacing every arc $(v_i, v_{i(i+1)}) \in V(C) \setminus V(D)$ with the path $(v_i, v_{i(i+1)}, v_{i+1})$ in $v_2 C v_j$.   And again, we construct the cycle $C' = P \cup (v_j, v_2)$ in $D$ of length $k+j-1$, where $k = |A(v_2 C v_j) \setminus A(D)|$.   Let us observe that $k \le j-2$, thus $k \le \frac{k+j-2}{2}$ and $k+1 \le \frac{k+j}{2}$. It follows from the main hypothesis that there are at least $\frac{k+j-1}{2} + 1$ symmetric arcs in $C'$.   Let us observe that, if the arc $(v_j, v_2)$ is symmetric in $D$, then there are at least $\frac{k+j-1}{2}$ symmetric arcs in $P$.   But this implies that either there is an arc of $A(v_2 C v_j) \cap A(D)$ that is symmetric in $D$, or that there exists $i$, with $2 \le i \le j-1$, such that the arcs $(v_i, v_{i(i+1)})$ and $(v_{i(i+1)}, v_{(i+1)})$ are symmetric in $D$.   In any case, $C$ has at least one symmetric arc.

\begin{figure}[ht]
\begin{center}
\begin{tikzpicture}


\node (1) at (-1,0)[circle, draw,inner sep=0pt, minimum width=5pt,label=180:$v_1$]{};
\node (2) at (0,0)[circle, draw,inner sep=0pt, minimum width=5pt,label=270:$v_2$]{};
\node (3) at (2,0)[circle, draw,inner sep=0pt, minimum width=5pt,label=0:$v_{j}$]{};


\node (P) at (1,-.5)[]{P};

\foreach \from/\to in {1/3}
\draw [->, shorten <=3pt, shorten >=3pt, >=stealth, line width=.7pt] (\from) to [bend left = 60] (\to);

\foreach \from/\to in {3/2}
\draw [->, shorten <=3pt, shorten >=3pt, >=stealth, line width=.7pt] (\from) to [bend right = 30] (\to);

\foreach \from/\to in {2/3}
\draw [->,decorate, decoration={snake,amplitude=.4mm,segment length=2mm,post length=1mm}, shorten <=0pt, shorten >=0pt, >=stealth, line width=.4pt] (\from) to [bend right = 15] (\to);

\end{tikzpicture}
\end{center}
\caption{The case $v_{12} = v_{j}$.}\label{Cf3}
\end{figure}
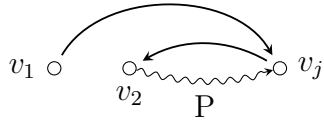

Since in any case the cycle $C$ has a symmetric arc, the result follows from the Principle of Mathematical Induction.

\end{proof}

Since every cycle of $\mathcal{C}^2(D)$ has a symmetric arc, we have that $D$ is kernel perfect due to Theorem \ref{Duchet}. By applying Lemma \ref{k-clo} to the digraph $D$ we get that $D$ has a $3$-kernel. This is stated in the following corollary
.
\begin{cor}
Let $D$ be a digraph. If every directed cycle $B$ in $D$ has at least $ \frac{1}{2} \left\| B \right\|  + 1$ symmetric arcs, then $D$ has a $3$-kernel.
\end{cor}

\section{$4$-kernels} \label{duc4}

Now, we will now prove a similar result for $4$-kernels. We need a few previous technical lemmas to do so.

\begin{obs}\label{obs1}

Let $D$ be a digraph. If every directed cycle $C$ in $D$ has at least $ \frac{2}{3} \left\| C \right\|  + 1$ symmetric arcs, then every cycle of length at most five is symmetric. Also, every cycle of length greater than five than has at least five symmetric arcs.

\end{obs}

\begin{lem} \label{lem1}
Let $D$ be a digraph such that every directed cycle $C$ in $D$ has at least $ \frac{2}{3} \left\| C \right\|  + 1$ symmetric arcs and $u,v \in V(D)$. If $P$ is a directed $uv$-path, $Q$ is a directed $vu$-path and $\left\|P\right\| + \left\|Q\right\| \leq 5$, then every arc in $A(P) \cup A(Q)$ is symmetric.
\end{lem}

\begin{proof}
 Clearly, if $\left\|P\right\| = 1$ or $\left\|Q\right\| = 1$, the result follows from Observation \ref{obs1}. Let $S_P = V(P) \setminus \left\{u,v\right\}$ and $S_Q = V(Q) \setminus \left\{u,v\right\}$. Suppose that $\left\|P\right\| = 2 =\left\|Q\right\|$. If $S_P \cap S_Q = \varnothing$, we have the desired result by Observation \ref{obs1}. If $S_P \cap S_Q \neq \varnothing$, then $Q$ is the path obtained by reversing the arcs of $P$, which means every arc in $A(P) \cup A(Q)$ is symmetric. 

Finally, assume without loss of generality that $\left\|P\right\| = 3$ and  $\left\|Q\right\| = 2$. Take $P = (u,x,y,v)$ and $Q = (v,z,u)$. If $z = x$, then $(u,x)$ is symmetric and $(x,y,v,x)$ is a $3$-cycle, which is symmetric by Observation \ref{obs1}. A similar argument works when $z = y$. In any case, every arc in $A(P) \cup A(Q)$ is symmetric.

\end{proof}

\begin{lem} \label{lem2}
Let $D$ be a digraph such that every directed cycle $C$ in $D$ has at least $ \frac{2}{3} \left\| C \right\|  + 1$ symmetric arcs and $u,v \in V(D)$. If $P$ is a directed $uv$-path, $Q$ is a directed $vu$-path and $\left\|P\right\| + \left\|Q\right\| \leq 6$, then every arc in $A(P) \cup A(Q)$ is symmetric but at most one.
\end{lem}

\begin{proof}

The cases where $\left\|P\right\| + \left\|Q\right\| \leq 5$ are covered by Lemma \ref{lem2}. Suppose that $\left\|P\right\| + \left\|Q\right\| = 6$. If $\left\|P\right\| = 1$ or $\left\|Q\right\| = 1$, then $PQ$ is a $C_6$ and has at least five symmetric arcs.  Take $S_P = V(P) \setminus \left\{u,v\right\}$ and $S_Q = V(Q) \setminus \left\{u,v\right\}$. If $S_P \cap S_Q = \varnothing$, we have that $uPvQu$ is a $C_6$ the result follows directly. We can thus assume that $S_P \cap S_Q \neq \varnothing$. If $\left\|P\right\| = 2$ and $\left\|Q\right\| = 4$, take $P = (u,w,v)$ and $Q = (v,x,y,z,u)$. If $w = x$, we have that the arc $(v,x)$ is symmetric and that $(u,x,y,z,u)$ is a $C_5$, so every arc in $A(P) \cup A(Q)$ is symmetric. If $w = z$, a similar argument yields the same result. If $w = y$, then $(u,y,z,u)$ and  $(v,x,y,v)$ are directed triangles, so they are symmetric and the results follows.

Finally, suppose that $\left\|P\right\| = 3 = \left\|Q\right\|$, $P = (u,z,w,v)$ and $Q = (v,x,y,u)$. If $z = y$ and $w = x$, then $Q$ is the path obtained by reversing the arrows of $P$, which means every arc in $A(P) \cup A(Q)$ is symmetric. If $z = x$ and $w = y$, then $(u,x,y,u)$ and $v,x,y,v$ are directed triangles, which are symmetric in $D$. If $z = y$ and $w \neq x$, then $(u,y)$ is symmetric and $(v,x,y,w,v)$ is a $C_4$, hence every arc in $A(P) \cup A(Q)$ is symmetric. Similar arguments solve the remaining cases.

\end{proof}

\begin{lem} \label{trian}
Let $D$ be a digraph. If every directed cycle $C$ in $D$ has at least $ \frac{2}{3} \left\| C \right\|  + 1$ symmetric arcs then, every $C_3$ in $H =\mathcal{C}^3 (D)$ has at least one symmetric arc.
\end{lem}

\begin{proof}
Let $C = (v_1,v_2,v_3,v_1)$ be a $3$-cycle of $H$. If $A(C) \subseteq A(D)$, then every arc of $C$ is symmetric. If $\left|A(C) \cap A(D)\right| = 2$, we can suppose that $(v_2,v_3), (v_3,v_1) \in A(D)$. Since $(v_1,v_2) \in A(H) \setminus A(D)$, there is a path $T$ of length at most three from $v_1$ to $v_2$. Either $v_3 \in V(T)$ or $v_3 \notin V(T)$. In any case, the arc $(v_3,v_1)$ is in a cycle in $D$ of length at most five and therefore is symmetric.

If $\left|A(C) \cap A(D)\right| = 1$, we can assume that $(v_3,v_1) \in A(D)$. Let $T_1$ be a $v_1v_2$-path of length at most three in $D$ and $T_2$ a $v_2v_3$-path of length at most three in $D$. If $v_3 \in V(T_1)$, then the arc $(v_3, v_1)$ is in a cycle in $D$ of length at most three, implying it is symmetric. If $v_1 \in V(T_2)$, then the arc $(v_2, v_1)$ is an arc of $H$, implying it is symmetric. Suppose that neither $v_3 \in V(T_1)$ nor $v_1 \in V(T_2)$. If $V(T_1) \cap V(T_2) = \{ v_2 \}$, then $T_1T_2 \cup (v_3, v_1)$ is a cycle of length at most $7$ in $D$ and has at least $6$ symmetric arcs, so either the arc $(v_3,v_1)$ is symmetric or both $(v_1,v_2)$ and $(v_2,v_3)$ are symmetric. If $V(T_1) \cap V(T_2) \neq \{ v_2 \}$, then the arc $(v_3,v_1)$ is in a cycle in $D$ of length at most five and it is therefore symmetric.

We can now assume that $A(C) \cap A(D) = \varnothing$. Let $T_i$ be the shortest $v_iv_{i+1}$-path of length at most three for $1 \leq i \leq n-1$ and take $S_i = V(T_i) \setminus \left\{v_i, v_{i+1}\right\}$. Also, let $T_n$ be the shortest $v_nv_{1}$-path of length at most three and take $S_n = V(T_n) \setminus \left\{v_1, v_{n}\right\}$. If $S_i \cap V(C) \neq \varnothing$ for some $i \in \left\{1,2,3\right\}$, then clearly $C$ has a symmetric arc. 

Suppose now that $S_i \cap S_j = \varnothing$ for every $i \neq j$. In this case, Lemma \ref{disjoint} gives us the existence of a symmetric arc of $C$. Finally, we must check what happens when there exist $i \neq j$ such that $S_i \cap S_j \neq \varnothing$. We can assume without loss of generality that $i = 1$ and $j = 2$. We must check all the different ways in which $T_1$ and $T_2$ can intersect. Notice that, since $S_1 \cap S_2 \neq \varnothing$, the distance from $v_1$ to $v_3$ is at most four. If $d(v_1,v_3) \leq 3$, then the arc $(v_1,v_3)$ is symmetric. It only remains to check when $d(v_1,v_3) = 4$. In this case, we can assume that $T_1 = (v_1, x_1, y_1, v_2)$ and $T_2 = (v_2, x_2, y_2, v_3)$, where $y_1 = x_2$ and the remaining vertices are all different. If $S_3 \cap \left(S_1 \cup S_2\right) = \varnothing$, then $A = v_1T_1y_1T_2v_3T_3v_1$ is a cycle of length six or seven. If its length is six, the main hypothesis implies that $A$ has at least five symmetric arcs. If its length is seven, the main hypothesis implies that $A$ has at least six symmetric arcs. In either case, either $(v_1,v_2)$ or $(v_2,v_3)$ is symmetric by the Pidgeonhole Principle.

Suppose that $S_3 \cap \left(S_1 \cup S_2\right) \neq \varnothing$. First, take $T_3 = (v_3, x_3, v_1)$. If $x_3 = x_2$, then $v_1T_1x_2T_3v_1$ is a $3$-cycle in $D$.   Hence, it is symmetric in $H$ and this means $(v_2,v_1) \in A(H)$, so it is a symmetric arc. If $x_3 = y_2$, then $v_1T_1x_2T_2y_2T_3v_1$ is a $4$-cycle in $D$, which is symmetric in $H$.   Thus, $(v_2,v_1) \in A(H)$, so it is a symmetric arc. If $x_3 = x_1$, an argument analogous to the one used in the previous case shows that $(v_3,v_2)$ is a symmetric arc of $H$.

Finally, take $T_3 = (v_3, x_3, y_3, v_1)$. If $y_3 \in S_1 \cup S_2$, arguments analogous to the ones used in the case where $T_3$ has length two work. Thus, we can assume that $y_2 \notin S_1 \cup S_2$ and $x_3 \in S_1 \cup S_2$. If $x_3 = x_2$, then $v_1T_1x_2T_3v_1$ is a $4$-cycle in $D$ and hence it is symmetric in $H$.   Therefore $(v_2,v_1) \in A(H)$, and $(v_1,v_2)$ is a symmetric arc. If $x_3 = y_2$, then $v_1T_1x_2T_2y_2T_3v_1$ is a $5$-cycle in $D$, which is symmetric in $H$.   So, $(v_2,v_1) \in A(H)$, and $(v_1,v_2)$ is symmetric. Again, if $x_3 = x_1$, an argument analogous to the previous one shows that $(v_3,v_2)$ is symmetric.

\end{proof}

Now, we can prove an analogue of Theorem \ref{3k} for $4$-kernels. It is not surprising, specially if one compares Lemma \ref{C_3} and Lemma \ref{trian}, that the basic structure of the proof of Theorem \ref{4k} is very similar to the one of Theorem \ref{3k}. Nevertheless, working with $4$-kernels means we have to work with longer paths in the digraph, which involves a few difficulties that are not present in the case of $3$-kernels.

We will work with a cycle in the $4$-closure of a digraph $D$ and the paths in $D$ that originate the arcs in that cycle. The proof consists of four main parts. First, we check what happens when all the paths are internally disjoint. This is easy thanks to Lemma \ref{disjoint}. 

After that, we start working assuming that two of those paths are not internally disjoint. We check what happens when those paths correspond to consecutive arcs in the cycle in the second part. A special case, which we will call an $\omega$-configuration, arises here.

Finally, we study the $\omega$-configurations along with the case where the paths correspond to arcs that are not consecutive in the cycle.

\begin{thm}\label{4k}
Let $D$ be a digraph. If every directed cycle $B$ in $D$ has at least $ \frac{2}{3} \left\| B \right\|  + 1$ symmetric arcs, then every cycle in $H = \mathcal{C}^3(D)$ has a symmetric arc.
\end{thm}

\begin{proof}

Let $C$ be a cycle in $H$. We proceed by induction on the length of $C$. The case when $C$ has length three is covered by Lemma \ref{trian}.

Suppose then that $C = (v_1,v_2, \dots, v_n,v_0)$ is an $n$-cycle in $H$. For every arc $(u,v) \in A(C)$ there is a directed $uv$-path in $D$ of length at most three (possibly the same arc). For every $1 \leq i \leq n-1$, let $T_i$ be such directed path and $T_n$ be the directed path from $v_n$ to $v_1$, and take $S_i = V(T_i) \setminus \left\{v_i, v_{i+1}\right\}$ for $1 \leq i \leq n-1$ and $S_n = V(T_n) \setminus \left\{v_1, v_n\right\}$. If $S_i \cap S_j = \varnothing$ and $S_i \cap V(C) = \varnothing$ for every $1 \leq i < j \leq n$, then Lemma \ref{disjoint} gives us the desired result.

If $S_i \cap V(C) \neq \varnothing$, then, for some $0 \leq j \leq n$, there is a $v_j \in S_i$. Without loss of generality, we can assume that $\left|j - i \right|$ is minimum with such property and that $i = 1$. This means $(v_1,v_j) \in A(H)$. If $j = n$, then $(v_1,v_n) \in A(H)$ and it is symmetric. If $j = 3$, then $(v_3,v_2) \in A(H)$ and it is symmetric. Hence, we can assume that $j \notin \left\{1,2,3,n\right\}$.

It is easy to see that $(v_1,v_j)$ and $(v_j,v_2)$ are arcs of $H$ (Figure \ref{Cf4}). The cycle $C' = v_1v_jCv_1$ is a cycle in $H$ of length less than $n$, so it has a symmetric arc by induction hypothesis. If such symmetric arc is an arc of $C$, we are done. Otherwise, there is a directed path $P$ of length at most three from $v_j$ to $v_1$. Since the length of the path $v_1T_1v_j$ is at most two, an application of Lemma \ref{lem1} with $P$ and $v_1T_1v_j$ gives us that the arcs of $v_1T_1v_j$ are symmetric. On the other hand, the cycle $C'' = v_2Cv_j \cup (v_j, v_2)$ is a cycle in $H$ of length less than $n$, hence it has a symmetric arc. We can assume that $(v_j,v_2)$ is the symmetric arc, otherwise we are done. Since $(v_j,v_2)$ is symmetric, there is a directed path $Q$ from $v_2$ to $v_j$ of length at most three. Again, applying Lemma \ref{lem1} to $Q$ and $v_jT_1v_2$ proves that the arcs of $v_jT_1v_2$ are symmetric. Since $T_1 = v_1T_1v_jT_1v_2$ and all its arcs are symmetric, we have that $(v_1,v_2)$ is a symmetric arc of $C$.

\begin{figure}[ht]
\begin{center}
\begin{tikzpicture}


\node (1) at (-1,0)[circle, draw,inner sep=0pt, minimum width=5pt,label=180:$v_1$]{};
\node (2) at (0,0)[circle, draw,inner sep=0pt, minimum width=5pt,label=180:$v_2$]{};
\node (3) at (2,0)[circle, draw,inner sep=0pt, minimum width=5pt,label=0:$v_{j}$]{};

\node (4) at (4,0)[circle, draw,inner sep=0pt, minimum width=5pt,label=180:$v_1$]{};
\node (5) at (5,0)[circle, draw,inner sep=0pt, minimum width=5pt,label=180:$v_2$]{};
\node (6) at (7,0)[circle, draw,inner sep=0pt, minimum width=5pt,label=0:$v_{j}$]{};
\node (7) at (6,.5)[circle, draw,inner sep=0pt, minimum width=5pt]{};

\node (8) at (9,0)[circle, draw,inner sep=0pt, minimum width=5pt,label=180:$v_1$]{};
\node (9) at (10,0)[circle, draw,inner sep=0pt, minimum width=5pt,label=180:$v_2$]{};
\node (10) at (12,0)[circle, draw,inner sep=0pt, minimum width=5pt,label=0:$v_{j}$]{};
\node (11) at (10.5,.8)[circle, draw,inner sep=0pt, minimum width=5pt]{};


\node (C) at (.5,-1.3)[]{$v_jCv_1$};
\node (C) at (5.5,-1.3)[]{$v_jCv_1$};
\node (C) at (10.5,-1.3)[]{$v_jCv_1$};

\node (C) at (1,-.5)[]{$v_2Cv_j$};
\node (C) at (6,-.5)[]{$v_2Cv_j$};
\node (C) at (11,-.5)[]{$v_2Cv_j$};


\foreach \from/\to in {1/3,4/6}
\draw [->, shorten <=3pt, shorten >=3pt, >=stealth, line width=.7pt] (\from) to [bend left = 60] (\to);

\foreach \from/\to in {3/2,10/9}
\draw [->, shorten <=3pt, shorten >=3pt, >=stealth, line width=.7pt] (\from) to [bend right = 30] (\to);

\foreach \from/\to in {6/7,7/5}
\draw [->, shorten <=3pt, shorten >=3pt, >=stealth, line width=.7pt] (\from) to [bend right = 0] (\to);

\foreach \from/\to in {2/3,5/6,9/10}
\draw [->,dashed, shorten <=0pt, shorten >=0pt, >=stealth, line width=.4pt] (\from) to [bend right = 20] (\to);

\foreach \from/\to in {3/1,6/4,10/8}
\draw [->, dashed, shorten <=3pt, shorten >=3pt, >=stealth, line width=.7pt] (\from) to [bend left = 90] (\to);

\foreach \from/\to in {8/11,11/10}
\draw [->, shorten <=3pt, shorten >=3pt, >=stealth, line width=.7pt] (\from) to [bend left = 25] (\to);

\end{tikzpicture}
\caption{The case $v_j \in S_1$.} \label{Cf4}
\end{center}
\end{figure}
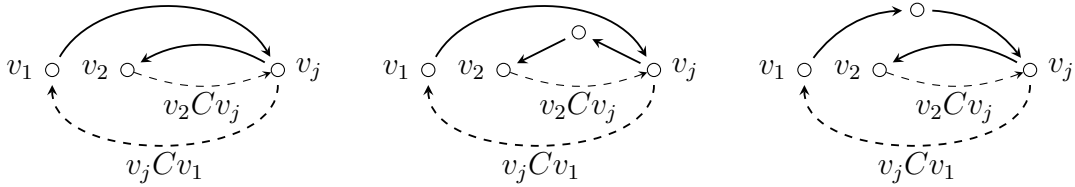

Suppose now that $S_i \cap V(C) = \varnothing$ for every $i \in \left\{1,2, \dots, n\right\}$ but $S_i \cap S_j \neq \varnothing$ for some $1 \leq i < j \leq n$. First, take the case where $v_i$ and $v_j$ are consecutive vertices in the cycle. Without loss of generality, take $i = 1$ and $j = 2$. We will check all the possible ways in which $T_1$ and $T_2$ can intersect. 

Notice that every time $d(v_1,v_3) \leq 3$ we will have that $(v_1,v_3) \in A(H)$. Hence, $(v_1, v_3) \cup v_3Cv_1$ is a cycle of length less than $n$ and the induction hypothesis gives us the existence of a symmetric arc which we can assume to be $(v_1,v_3)$. This is because we would have a symmetric arc of $C$ otherwise. The fact that $(v_1,v_3)$ is symmetric implies there is a path $P$ of length at most $3$ from $v_3$ to $v_1$. We will use this fact whenever we can in all the following cases.

\begin{itemize}

\item $\left\|T_1\right\| = 2 = \left\|T_2\right\|$ (Figure \ref{Cf5}). The only possible way in which they can intersect is when $T_1 = (v_1,x,v_2)$ and $T_2 = (v_2,x,v_3)$. In this case, Lemma \ref{lem1} implies that $(v_1,x)$ is symmetric, so the arc $v_1,v_2$ is a symmetric arc of $C$.

\begin{figure}[ht]
\begin{center}
\begin{tikzpicture}


\node (1) at (-2,0)[circle, draw,inner sep=0pt, minimum width=5pt,label=180:$v_1$]{};
\node (2) at (0,0)[circle, draw,inner sep=0pt, minimum width=5pt,label=270:$v_2$]{};
\node (3) at (2,0)[circle, draw,inner sep=0pt, minimum width=5pt,label=0:$v_{3}$]{};

\node (12) at (0,1.5)[circle, draw,inner sep=0pt, minimum width=5pt,label=90:$x$]{};

\foreach \from/\to in {1/12,12/3}
\draw [->, shorten <=3pt, shorten >=3pt, >=stealth, line width=.7pt] (\from) to [bend right = 0] (\to);

\foreach \from/\to in {12/2,2/12}
\draw [->, shorten <=0pt, shorten >=0pt, >=stealth, line width=.4pt] (\from) to [bend right = 20] (\to);

\end{tikzpicture}
\caption{$\left\|T_1\right\| = 2 = \left\|T_2\right\|$.} \label{Cf5}
\end{center}
\end{figure}

\pagebreak

\item $\left\|T_1\right\| = 2$ and $\left\|T_2\right\| = 3$ (Figure \ref{Cf6}). Let $T_1 = (v_1,x,v_2)$ and $T_2 = (v_2,y,z,v_3)$. 

If $y = x$, then take $Q = (v_1,x,z,v_3)$. By applying Lemma \ref{lem2} we derive that either $(v_1,x)$ is symmetric, in which case $(v_1,v_2)$ is a symmetric arc of $C$, or both $(x,z)$ and $(z,v_3)$ are symmetric, implying that $(v_2,v_3)$ is a symmetric arc of $C$.

If $z = x$, then take $Q = (v_1,z,v_3)$. Here, Lemma \ref{lem1} gives us that $(v_1,x)$ is symmetric, $(v_2,y,x,v_1)$ is a path of length $3$ and therefore the arc $(v_1,v_2)$ is symmetric.

\begin{figure}[ht]
\begin{center}
\begin{tikzpicture}


\node (1) at (-2,0)[circle, draw,inner sep=0pt, minimum width=5pt,label=180:$v_1$]{};
\node (2) at (0,0)[circle, draw,inner sep=0pt, minimum width=5pt,label=270:$v_2$]{};
\node (3) at (2,0)[circle, draw,inner sep=0pt, minimum width=5pt,label=0:$v_{3}$]{};
\node (4) at (1,1.5)[circle, draw,inner sep=0pt, minimum width=5pt,label=0:$z$]{};

\node (12) at (-1,1.5)[circle, draw,inner sep=0pt, minimum width=5pt,label=90:$x$]{};

\node (5) at (5,0)[circle, draw,inner sep=0pt, minimum width=5pt,label=180:$v_1$]{};
\node (6) at (7,0)[circle, draw,inner sep=0pt, minimum width=5pt,label=270:$v_2$]{};
\node (7) at (9,0)[circle, draw,inner sep=0pt, minimum width=5pt,label=0:$v_{3}$]{};
\node (8) at (7,1.5)[circle, draw,inner sep=0pt, minimum width=5pt,label=180:$x$]{};

\node (9) at (8,.7)[circle, draw,inner sep=0pt, minimum width=5pt,label=270:$y$]{};

\foreach \from/\to in {1/12,12/4,4/3,5/8,8/6,6/9,9/8}
\draw [->, shorten <=3pt, shorten >=3pt, >=stealth, line width=.7pt] (\from) to [bend right = 0] (\to);

\foreach \from/\to in {12/2,2/12}
\draw [->, shorten <=1pt, shorten >=1pt, >=stealth, line width=.4pt] (\from) to [bend right = 20] (\to);

\foreach \from/\to in {8/7}
\draw [->, shorten <=1pt, shorten >=1pt, >=stealth, line width=.4pt] (\from) to [bend left = 30] (\to);

\end{tikzpicture}
\caption{$\left\|T_1\right\| = 2$ and $\left\|T_2\right\| = 3$.} \label{Cf6}
\end{center}
\end{figure}
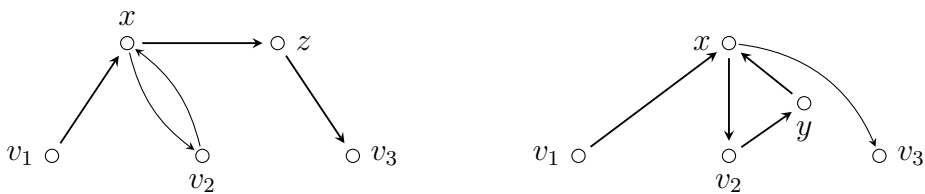

\item $\left\|T_1\right\| = 3$ and $\left\|T_2\right\| = 2$. This is very similar to the previous case. Let $T_1 = (v_1,x,y,v_2)$ and $T_2 = (v_2,z,v_3)$. 

If $z = y$, take $Q = (v_1,x,y,v_3)$. Using Lemma \ref{lem1} we can see that either both arcs in $(v_1,x,y)$ are symmetric, and hence the arc $(v_1,v_2)$ is a symmetric arc of $C$, or $(z,v_3)$ is symmetric, implying $(v_2,v_3)$ is a symmetric arc in $C$.

If $z = x$, then $Q = (v_1,x,v_3)$. Lemma \ref{lem1} guarantees that every arc in $Q$ is symmetric, so $(v_2,x,v_1)$ is a directed path in $D$ and, therefore, the arc $(v_1,v_2)$ is symmetric.

\item $\left\|T_1\right\| = 3 = \left\|T_2\right\|$. Let $T_1 = (v_1,x,y,v_2)$ and $T_2 = (v_2,z,w,v_3)$. 

If $z = x$ and $w = y$ (Figure \ref{Cf7} (a)), take $Q = (v_1,x,y,v_3)$. Now, Lemma \ref{lem2} guarantees that either $(v_1,x)$ or $(y,v_3)$ is symmetric. In the first case we have that $(v_2,x,v_1)$ is a path in $D$ and $(v_1,v_2)$ is symmetric. In the second case, $(v_3,y,v_2)$ is a path in $D$ and $(v_1,v_2)$ is symmetric.

If $z = y$ and $w = x$ (Figure \ref{Cf7} (b)), take $Q = (v_1,x,v_3)$ and apply Lemma \ref{lem1}. We get that $(v_1,x)$ is symmetric and this means $(v_1,v_2)$ is a symmetric arc of $C$.

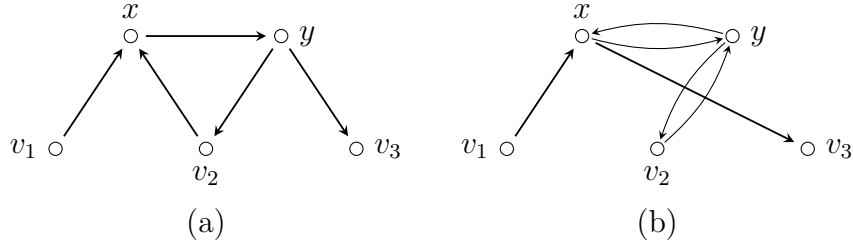
\begin{figure}[ht]
\begin{center}
\begin{tikzpicture}


\node (1) at (-2,0)[circle, draw,inner sep=0pt, minimum width=5pt,label=180:$v_1$]{};
\node (2) at (0,0)[circle, draw,inner sep=0pt, minimum width=5pt,label=270:$v_2$]{};
\node (3) at (2,0)[circle, draw,inner sep=0pt, minimum width=5pt,label=0:$v_{3}$]{};
\node (4) at (-1,1.5)[circle, draw,inner sep=0pt, minimum width=5pt,label=90:$x$]{};
\node (5) at (1,1.5)[circle, draw,inner sep=0pt, minimum width=5pt,label=0:$y$]{};

\node (C) at (0,-1)[]{(a)};
\node (P) at (6,-1)[]{(b)};

\node (6) at (4,0)[circle, draw,inner sep=0pt, minimum width=5pt,label=180:$v_1$]{};
\node (7) at (6,0)[circle, draw,inner sep=0pt, minimum width=5pt,label=270:$v_2$]{};
\node (8) at (8,0)[circle, draw,inner sep=0pt, minimum width=5pt,label=0:$v_{3}$]{};
\node (9) at (5,1.5)[circle, draw,inner sep=0pt, minimum width=5pt,label=90:$x$]{};
\node (10) at (7,1.5)[circle, draw,inner sep=0pt, minimum width=5pt,label=0:$y$]{};

\foreach \from/\to in {1/4,4/5,5/2,2/4,5/3,6/9,9/8}
\draw [->, shorten <=3pt, shorten >=3pt, >=stealth, line width=.7pt] (\from) to [bend right = 0] (\to);

\foreach \from/\to in {7/10,10/7,9/10,10/9}
\draw [->, shorten <=1pt, shorten >=1pt, >=stealth, line width=.4pt] (\from) to [bend right = 15] (\to);

\foreach \from/\to in {}
\draw [->, shorten <=1pt, shorten >=1pt, >=stealth, line width=.4pt] (\from) to [bend left = 30] (\to);

\end{tikzpicture}
\caption{$\left\|T_1\right\| = 2$ and $\left\|T_2\right\| = 3$.} \label{Cf7}
\end{center}
\end{figure}

If $z = x$ and $w \neq y$ (Figure \ref{Cf8} (a)), take $Q = (v_1,x,w,v_3)$ . Notice that $(x,y,v_2,x)$ is a $C_3$ of $D$, so it is symmetric. Applying now Lemma \ref{lem2} gives us that either $(v_1,x)$ is symmetric, implying that $(v_1,v_2)$ is a symmetric arc of $C$, or both $(z,w)$ and $(w,v_3)$ are symmetric, hence $(v_2,v_3)$ is a symmetric arc of $C$.

If $z \neq x$ and $w = y$ (Figure \ref{Cf8} (b)), take $Q = (v_1,x,y,v_3)$. In a way similar to the previous case, $(v_2,z,y,v_2)$ is a $C_3$ of $D$, so it is symmetric. An application of Lemma \ref{lem2} gives us that either $(v_1,v_2)$ or  $(v_2,v_3)$ is a symmetric arc of $C$.

If $z \neq y$ and $w = x$ (Figure \ref{Cf8} (c)), we have $Q = (v_1,x,v_3)$. Here, an application of Lemma \ref{lem1} gives us that $(v_1,x)$ is symmetric. This means that $(v_2,z,x,v_1)$ is a path in $D$ and therefore $(v_1,v_2)$ is symmetric.

\begin{figure}[ht]
\begin{center}
\begin{tikzpicture}


\node (1) at (-2,0)[circle, draw,inner sep=0pt, minimum width=5pt,label=180:$v_1$]{};
\node (2) at (0,0)[circle, draw,inner sep=0pt, minimum width=5pt,label=270:$v_2$]{};
\node (3) at (2,0)[circle, draw,inner sep=0pt, minimum width=5pt,label=0:$v_{3}$]{};
\node (4) at (-1,1.5)[circle, draw,inner sep=0pt, minimum width=5pt,label=90:$x$]{};
\node (5) at (1,.75)[circle, draw,inner sep=0pt, minimum width=5pt,label=270:$y$]{};
\node (6) at (1,1.5)[circle, draw,inner sep=0pt, minimum width=5pt,label=0:$w$]{};

\node (C) at (0,-1)[]{(a)};
\node (P) at (6,-1)[]{(b)};
\node (P) at (3,-5)[]{(c)};

\node (7) at (4,0)[circle, draw,inner sep=0pt, minimum width=5pt,label=180:$v_1$]{};
\node (8) at (6,0)[circle, draw,inner sep=0pt, minimum width=5pt,label=270:$v_2$]{};
\node (9) at (8,0)[circle, draw,inner sep=0pt, minimum width=5pt,label=0:$v_{3}$]{};
\node (10) at (5,1.5)[circle, draw,inner sep=0pt, minimum width=5pt,label=90:$x$]{};
\node (11) at (7,1.5)[circle, draw,inner sep=0pt, minimum width=5pt,label=90:$y$]{};
\node (12) at (5,.75)[circle, draw,inner sep=0pt, minimum width=5pt,label=270:$z$]{};

\node (13) at (1,-4)[circle, draw,inner sep=0pt, minimum width=5pt,label=180:$v_1$]{};
\node (14) at (3,-4)[circle, draw,inner sep=0pt, minimum width=5pt,label=270:$v_2$]{};
\node (15) at (5,-4)[circle, draw,inner sep=0pt, minimum width=5pt,label=0:$v_{3}$]{};
\node (16) at (3,-2)[circle, draw,inner sep=0pt, minimum width=5pt,label=90:$x$]{};
\node (17) at (3.5,-3)[circle, draw,inner sep=0pt, minimum width=5pt,label=290:$y$]{};
\node (18) at (2.5,-3)[circle, draw,inner sep=0pt, minimum width=5pt,label=250:$z$]{};

\foreach \from/\to in {1/4,4/6,5/2,2/4,6/3,4/5,7/10,10/11,11/8,8/12,12/11,11/9,13/16,16/17,17/14,14/18,18/16}
\draw [->, shorten <=3pt, shorten >=3pt, >=stealth, line width=.7pt] (\from) to [bend right = 0] (\to);

\foreach \from/\to in {16/15}
\draw [->, shorten <=3pt, shorten >=3pt, >=stealth, line width=.7pt] (\from) to [bend left = 0] (\to);

\end{tikzpicture}
\caption{$\left\|T_1\right\| = 3 = \left\|T_2\right\|$} \label{Cf8}
\end{center}
\end{figure}

\end{itemize}

The only remaining case is when $T_1 = (v_1,x,y,v_2)$,  $T_2 = (v_2,z,w,v_3)$ and we have $z = y$ and $w \neq x$. Let use call this case an $\omega$-configuration and say that $T_1$ and $T_2$ intersect in an $\omega$-configuration. The arcs $(x,y)$ and $(y,w)$ will be called the \emph{inner arcs} of the $\omega$-configuration formed by $T_1$ and $T_2$, and we will use $\iota (T_1, T_2)$ to denote the set $\left\{(x,y), (y,w)\right\}$. The symmetric arc $(v_2,y)$ will be called the \emph{spike} of the $\omega$-configuration (see Figure \ref{Cf9}) formed by $T_1$ and $T_2$ and we will use $\sigma(T_1,T_2)$ to denote the set $\left\{(v_2,y),(y,v_2)\right\}$. The arcs $(v_1,x)$ and $(w,v_3)$ will be called the \emph{outer} arcs and the set $\left\{(v_1,x),(w,v_3)\right\}$ will be denoted by $\epsilon(T_1,T_2)$.

\begin{figure}[ht]
\begin{center}
\begin{tikzpicture}


\node (1) at (-2,0)[circle, draw,inner sep=0pt, minimum width=5pt,label=270:$v_1$]{};
\node (2) at (0,0)[circle, draw,inner sep=0pt, minimum width=5pt,label=270:$v_2$]{};
\node (3) at (2,0)[circle, draw,inner sep=0pt, minimum width=5pt,label=270:$v_{3}$]{};
\node (4) at (-2,1)[circle, draw,inner sep=0pt, minimum width=5pt,label=90:$x$]{};
\node (5) at (0,1)[circle, draw,inner sep=0pt, minimum width=5pt,label=90:$y$]{};
\node (6) at (2,1)[circle, draw,inner sep=0pt, minimum width=5pt,label=90:$w$]{};

\foreach \from/\to in {1/4,4/5,5/6,6/3}
\draw [->, shorten <=3pt, shorten >=3pt, >=stealth, line width=.7pt] (\from) to [bend right = 0] (\to);

\foreach \from/\to in {5/2,2/5}
\draw [->, shorten <=1pt, shorten >=1pt, >=stealth, line width=.7pt] (\from) to [bend left = 30] (\to);

\end{tikzpicture}
\caption{The $\omega$-configuration.} \label{Cf9}
\end{center}
\end{figure}
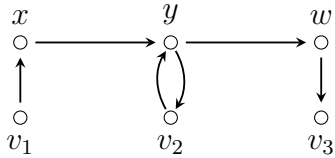

Since in this case we do not have a directed path from $v_1$ to $v_3$ of length less than $3$, we must proceed in a different manner. We can assume that whenever there are $T_i$ and $T_j$ such that $S_i \cap S_j \neq \varnothing$ either there are $i \leq l_1 < l_2 \leq j$ such that $S_{l_1} \cap S_{l_2} \neq \varnothing$ with $i \neq l_1$ or $j \neq l_2$, or $v_i$ and $v_j$ are consecutive in $C$ and the intersection between $T_i$ and $T_j$ is an $\omega$-configuration.

Let $I = \left\{v_k, v_{k+1}, \dots, v_{k+t}\right\}$ a set of consecutive vertices of the cycle $C$, where the subscripts are taken in the natural way induced by the cycle. We say that $I$ is a $\omega$-block if the following conditions are satisfied:

\begin{enumerate}

\item $(v_r,v_{r+1}) \notin A(D)$. 

\item The intersection between $T_r$ and $T_{r+1}$ is a $\omega$-configuration for every $k \leq r \leq k+t$. 

\item Either $(v_{k-1},v_{k})  \in A(D)$ or $S_{k-1} \cap S_k = \varnothing$.

\item Either $(v_{k+t},v_{k+t+1})  \in A(D)$ or $S_{k+t} \cap S_{k+t+1} = \varnothing$.

\end{enumerate}

An $\omega$-block $I = \left\{v_k, v_{k+1}, \dots, v_{k+t}\right\}$ will be called proper if $S_i \cap S_j = \varnothing$ when $i$ and $j$ are not consecutive, like in Figure \ref{Cf10}. Otherwise, $I$ will be called improper. An example of an improper $\omega$-block can be seen in Figure \ref{Cf11}. Clearly, improper $\omega$-blocks have at least four vertices.

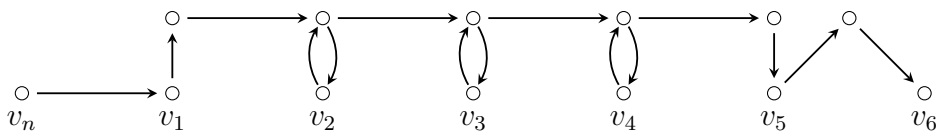
\begin{figure}[ht]
\begin{center}
\begin{tikzpicture}


\node (1) at (-2,0)[circle, draw,inner sep=0pt, minimum width=5pt,label=270:$v_1$]{};
\node (2) at (0,0)[circle, draw,inner sep=0pt, minimum width=5pt,label=270:$v_2$]{};
\node (3) at (2,0)[circle, draw,inner sep=0pt, minimum width=5pt,label=270:$v_{3}$]{};
\node (4) at (-2,1)[circle, draw,inner sep=0pt, minimum width=5pt]{};
\node (5) at (0,1)[circle, draw,inner sep=0pt, minimum width=5pt]{};
\node (6) at (2,1)[circle, draw,inner sep=0pt, minimum width=5pt]{};
\node (7) at (4,0)[circle, draw,inner sep=0pt, minimum width=5pt,label=270:$v_4$]{};
\node (8) at (6,0)[circle, draw,inner sep=0pt, minimum width=5pt,label=270:$v_{5}$]{};
\node (9) at (4,1)[circle, draw,inner sep=0pt, minimum width=5pt]{};
\node (10) at (6,1)[circle, draw,inner sep=0pt, minimum width=5pt]{};
\node (a) at (-4,0)[circle, draw,inner sep=0pt, minimum width=5pt,label=270:$v_n$]{};
\node (b) at (8,0)[circle, draw,inner sep=0pt, minimum width=5pt,label=270:$v_6$]{};
\node (c) at (7,1)[circle, draw,inner sep=0pt, minimum width=5pt]{};

\foreach \from/\to in {1/4,4/5,5/6,6/9,9/10,10/8,8/c,c/b,a/1}
\draw [->, shorten <=3pt, shorten >=3pt, >=stealth, line width=.7pt] (\from) to [bend right = 0] (\to);

\foreach \from/\to in {5/2,2/5,6/3,3/6,9/7,7/9}
\draw [->, shorten <=1pt, shorten >=1pt, >=stealth, line width=.7pt] (\from) to [bend left = 30] (\to);

\end{tikzpicture}
\caption{The set $\left\{v_1,v_2,v_3,v_4,v_5\right\}$ is a proper $\omega$-block.} \label{Cf10}
\end{center}
\end{figure}

\begin{figure}[ht]
\begin{center}
\begin{tikzpicture}


\node (1) at (-2,0)[circle, draw,inner sep=0pt, minimum width=5pt,label=270:$v_1$]{};
\node (2) at (0,0)[circle, draw,inner sep=0pt, minimum width=5pt,label=270:$v_2$]{};
\node (3) at (2,0)[circle, draw,inner sep=0pt, minimum width=5pt,label=270:$v_{3}$]{};
\node (4) at (-2,1)[circle, draw,inner sep=0pt, minimum width=5pt]{};
\node (5) at (0,1)[circle, draw,inner sep=0pt, minimum width=5pt,label=90:$a$]{};
\node (6) at (2,1)[circle, draw,inner sep=0pt, minimum width=5pt,label=90:$b$]{};
\node (7) at (4,0)[circle, draw,inner sep=0pt, minimum width=5pt,label=270:$v_4$]{};
\node (8) at (6,0)[circle, draw,inner sep=0pt, minimum width=5pt,label=270:$v_{5}$]{};
\node (9) at (4,1)[circle, draw,inner sep=0pt, minimum width=5pt,label=90:$c$]{};
\node (a) at (-4,0)[circle, draw,inner sep=0pt, minimum width=5pt,label=270:$v_n$]{};
\node (b) at (8,0)[circle, draw,inner sep=0pt, minimum width=5pt,label=270:$v_6$]{};
\node (c) at (7,1)[circle, draw,inner sep=0pt, minimum width=5pt]{};


\foreach \from/\to in {1/4,4/5,5/6,6/9,8/c,c/b,a/1}
\draw [->, shorten <=3pt, shorten >=3pt, >=stealth, line width=.7pt] (\from) to [bend right = 0] (\to);

\foreach \from/\to in {5/2,2/5,6/3,3/6,9/7,7/9}
\draw [->, shorten <=1pt, shorten >=1pt, >=stealth, line width=.7pt] (\from) to [bend left = 30] (\to);

\foreach \from/\to in {5/8}
\draw [->, shorten <=1pt, shorten >=1pt, >=stealth, line width=.7pt] (\from) to [bend left = 60] (\to);

\foreach \from/\to in {9/5}
\draw [->, shorten <=3pt, shorten >=3pt, >=stealth, line width=.7pt] (\from) to [bend right = 30] (\to);

\end{tikzpicture}
\caption{The set $\left\{v_1,v_2,v_3,v_4,v_5\right\}$ is an improper $\omega$-block.} \label{Cf11}
\end{center}
\end{figure}
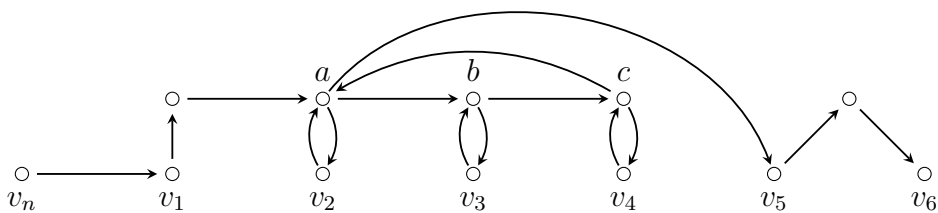

Let $I = \left\{v_k, v_{k+1}, \dots, v_{k+t}\right\}$ be an improper $\omega$-block. This means that there are integers $k_1, k_2$ such that $k \leq k_1 < k_1 + 1 < k_2 \leq k+t$ and $S_{k_1}\cap S_{k_2} \neq \varnothing$. We can assume that $k_2 - k_1$ is minimum with such property and that $k_1 = 1$ and $k_2 = j$. Let $T_1 = (v_1,x,y,v_2)$ and $T_j = (v_j,z,w,v_{j+1})$. Since $z \in T_{j-1}$, the minimality of $j - 1$ guarantees that $z \notin \left\{x,y\right\}$. This means that either $w = x$ or $w = y$.

If $w = x$, take the cycle $B$ in $D$ that is induced by the arc set $ \left\{(z,x)\right\} \cup E$, where $$E = \bigcup ^{j-1}_{r = 1} \iota(T_r,T_{r+1}). $$

Since $2 < j$, we have that $E \neq \varnothing$. In Figure \ref{Cf11}, the cycle $(a,b,c,a)$ is the cycle $B$. If $\left\| (B) \right\| \leq 5$, then $B$ is symmetric and and every arc in $E$ is symmetric, so the arc $(v_2,v_3)$ is a symmetric arc in $C$. If $\left\| B \right\| \geq 6$, then it has at least five symmetric arcs. Clearly $\left|A(B) \setminus E\right| = 1$, so there is a symmetric arc in $E$ and, therefore, a symmetric arc in $C$. The case where $x = y$ is analogous.

From now on, we can assume that if two consecutive vertices $v_i,v_{i+1}$ satisfy $S_i \cap S_{i+1} \neq \varnothing$, then $T_i$ and $T_{i+1}$ intersect in a $\omega$-configuration and there exists an $\omega$-block $I$ such that $v_i,v_{i+1} \in I$. Also, we can suppose that every $\omega$-block is proper.

First, suppose that whenever there are $T_i$ and $T_j$ such that $S_i \cap S_j \neq \varnothing$, we have $i + 1 = j$ (with indices taken in the natural way along the cycle). Clearly, this means that for every arc $(v_i, v_{i+1}) \in A(C)$, exactly one of the following conditions is fulfilled:

\begin{itemize}

\item $(v_i, v_{i+1}) \in A(D)$.

\item $S_i \cap S_j$ for every $j \neq i$.

\item There is an $\omega$-block $I$ such that $v_i,v_{i+1} \in I$.

\end{itemize}

Let $\mathfrak{L}, \mathfrak{K}, \Omega, \mathfrak{O}$ and $\alpha$ be defined as follows:

\begin{itemize}

\item $\mathfrak{L} = \left|\left\{T_i: 2 \leq i \leq j, \left\|T_i\right\| = 3\right\}\right|$.

\item $\mathfrak{K} = \left|\left\{T_i: 2 \leq i \leq j, \left\|T_i\right\| = 2\right\}\right|$.

\item $\Omega$ is the set of all the $\omega$-blocks contained in $V(C)$.

\item $\mathfrak{O} = \left|\Omega\right|$.

\item $\alpha = \left\{(x,y) \in A(C) \setminus A(D): \left\{x,y\right\} \not\subset I, I \in \Omega \right\}$.

\end{itemize}

It is easy to see that $B$, the cycle induced by $$\left[A(C)\cap A(D)\right] \cup \left[\bigcup_{(v_i,v_{i+1}) \in \alpha} A(T_i)\right] \cup \left[\bigcup_{I \in \Omega} \iota_I\right] \cup \left[\bigcup_{I \in \Omega} \epsilon_I\right],$$ has length $n + 2\mathfrak{O} + 2\mathfrak{L} + \mathfrak{K}$.

Since $\mathfrak{O} + \mathfrak{L} \leq n$ and $\mathfrak{O} + \mathfrak{L} + \mathfrak{K} \leq n$, we have the following inequalities:

\begin{center}
\begin{eqnarray*}
2\mathfrak{O} + 2\mathfrak{L} + \mathfrak{K}&\leq& 2n\\
6\mathfrak{O} + 6\mathfrak{L} + 3\mathfrak{K}&\leq& 2n + 4\mathfrak{O} + 4\mathfrak{L} + 2\mathfrak{K}\\
2\mathfrak{O} + 2\mathfrak{L} + \mathfrak{K}&\leq& \frac{2}{3}\left(n + 2\mathfrak{O} + 2\mathfrak{L} + \mathfrak{K}\right)
	\end{eqnarray*}
\end{center}

The main hypothesis implies that $B$ has at least $\frac{2}{3}\left(n + 2\mathfrak{O} + 2\mathfrak{L} + \mathfrak{K}\right) + 1$ symmetric arcs, so $C$ has at least one symmetric arc.

Suppose now that there are positive integers $i,j$ such that $S_i \cap S_j \neq \varnothing$ and $1 \leq i < j < n$. Again, we must check every way in which $S_i$ and $S_j$ may intersect. Suppose that $j-i$ is minimum with these properties and, without loss of generality, that $i = 1$. First, let us check the most direct cases.

\begin{enumerate}

\item $\left\|T_1\right\| = 2 = \left\|T_j\right\|$ (Figure \ref{Cf12}). The only possible way in which they can intersect is when $T_1 = (v_1,x,v_2)$ and $T_j = (v_j,x,v_{j+1})$. Let $P_1 = (v_1,x,v_{j+1})$ and $P_2 = (v_j,x,v_2)$. Clearly, the arcs $(v_j,v_2)$ and $(v_1,v_{j+1})$ are arcs of $H$. Take $B_1 = (v_1, v_{j+1}) \cup v_{j+1}Cv_1$ and $B_2 = v_2Cv_{j+1} \cup (v_{j+1}, v_2)$. Clearly, $B_1$ and $B_2$ have length less than $n$ and, by induction hypothesis, they have a symmetric arc. We can assume that the symmetric arc in $B_1$ is $(v_1,v_{j+1})$ and the one in $B_2$ is $(v_j,v_{2})$, since otherwise we would have a symmetric arc in $C$. This means that there is a directed $v_{j+1}v_1$-path and a directed $v_{2}v_j$-path in $D$. Let us call them $Q_1$ and $Q_2$ respectively. Now, applying Lemma \ref{lem1} gives us that the arcs in $T_1$ and $T_j$ are symmetric, so $(v_1,v_2)$ and $(v_j,v_{j+1})$ are symmetric arcs in $C$.

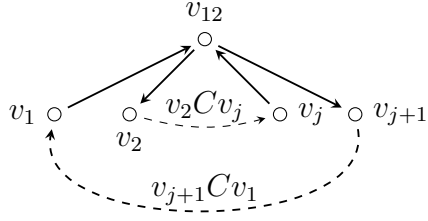
\begin{figure}[ht]
\begin{center}
\begin{tikzpicture}


\node (1) at (-1,0)[circle, draw,inner sep=0pt, minimum width=5pt,label=180:$v_1$]{};
\node (2) at (0,0)[circle, draw,inner sep=0pt, minimum width=5pt,label=270:$v_2$]{};
\node (3) at (2,0)[circle, draw,inner sep=0pt, minimum width=5pt,label=0:$v_{j}$]{};
\node (4) at (3,0)[circle, draw,inner sep=0pt, minimum width=5pt,label=0:$v_{j+1}$]{};

\node (12) at (1,1)[circle, draw,inner sep=0pt, minimum width=5pt,label=90:$v_{12}$]{};
\node (C) at (1,-1)[]{$v_{j+1}Cv_1$};
\node (P) at (1,.1)[]{$v_{2}Cv_j$};

\foreach \from/\to in {1/12,12/2,3/12,12/4}
\draw [->, shorten <=3pt, shorten >=3pt, >=stealth, line width=.7pt] (\from) to [bend right = 0] (\to);

\foreach \from/\to in {2/3}
\draw [->,dashed, shorten <=3pt, shorten >=3pt, >=stealth, line width=.4pt] (\from) to [bend right = 15] (\to);

\foreach \from/\to in {4/1}
\draw [->,dashed, shorten <=3pt, shorten >=3pt, >=stealth, line width=.7pt] (\from) to [bend left = 100] (\to);

\end{tikzpicture}
\caption{The case $\left\|T_1\right\| = 2 = \left\|T_j\right\|$.} \label{Cf12}
\end{center}
\end{figure}

\item If $\left\|T_1\right\| = 2$ and $\left\|T_j\right\| = 3$. Let $T_1 = (v_1,x,v_2)$ and $T_j = (v_j,y,z,v_{j+1})$. 

If $x = y$  (see Figure \ref{Cf13} (a)), then we can see that the arcs $(v_j,x)$ and $(x,v_2)$ are symmetric arcs of $D$ just like in the previous case. Now, since the arc $(v_1,v_{j+1}) \in A(H)$, we have that $B_1 = v_1v_{j+1}Cv_1$ is a cycle of length less than $n$, so it has a symmetric arc and we can assume it is $(v_1,v_{j+1})$. Hence, there is a path $Q_1$ from $v_{j+1}$ to $v_1$ of length at most three. By applying Lemma \ref{lem2}, we get that either $(v_1,x)$ is a symmetric arc of $D$ and so $(v_1,v_2)$ is a symmetric arc in $C$, or both $(x,y)$ and $(y,v_{j+1})$ are symmetric arcs of $D$ and thus $(v_j,v{j+1})$ is a symmetric arc in $C$. The case $x = z$  (Figure \ref{Cf13} (b)) is very similar.

\begin{figure}[ht]
\begin{center}
\begin{tikzpicture}


\node (1) at (-1,0)[circle, draw,inner sep=0pt, minimum width=5pt,label=180:$v_1$]{};
\node (2) at (0,0)[circle, draw,inner sep=0pt, minimum width=5pt,label=180:$v_2$]{};
\node (3) at (1.3,0)[circle, draw,inner sep=0pt, minimum width=5pt,label=0:$v_{j}$]{};
\node (4) at (.65,1.5)[circle, draw,inner sep=0pt, minimum width=5pt,label=90:$x$]{};
\node (5) at (1.475,.75)[circle, draw,inner sep=0pt, minimum width=5pt,label=0:$z$]{};
\node (6) at (2.3,0)[circle, draw,inner sep=0pt, minimum width=5pt,label=0:$v_{j+1}$]{};

\node (C) at (.65,-1.5)[]{(a)};
\node (P) at (6.65,-1.5)[]{(b)};

\node (7) at (5,0)[circle, draw,inner sep=0pt, minimum width=5pt,label=180:$v_1$]{};
\node (8) at (6,0)[circle, draw,inner sep=0pt, minimum width=5pt,label=180:$v_2$]{};
\node (9) at (7.3,0)[circle, draw,inner sep=0pt, minimum width=5pt,label=0:$v_{j}$]{};
\node (10) at (6.65,1.5)[circle, draw,inner sep=0pt, minimum width=5pt,label=90:$x$]{};
\node (11) at (6.975,.75)[circle, draw,inner sep=0pt, minimum width=5pt,label=180:$y$]{};
\node (12) at (8.3,0)[circle, draw,inner sep=0pt, minimum width=5pt,label=0:$v_{j+1}$]{};

\node (C) at (.65,-.7)[]{$v_{j+1}Cv_1$};
\node (P) at (6.65,-.7)[]{$v_{j+1}Cv_1$};

\foreach \from/\to in {1/4,4/2,3/4,4/5,5/6,7/10,10/8,9/11,11/10}
\draw [->, shorten <=3pt, shorten >=3pt, >=stealth, line width=.7pt] (\from) to [bend right = 0] (\to);

\foreach \from/\to in {10/12}
\draw [->, shorten <=1pt, shorten >=1pt, >=stealth, line width=.4pt] (\from) to [bend left = 0] (\to);

\foreach \from/\to in {6/1,12/7}
\draw [->, dashed, shorten <=1pt, shorten >=1pt, >=stealth, line width=.4pt] (\from) to [bend left = 80] (\to);

\foreach \from/\to in {2/3,8/9}
\draw [->, decorate, decoration={snake,amplitude=.4mm,segment length=2mm,post length=1mm}, shorten <=1pt, shorten >=1pt, >=stealth, line width=.4pt] (\from) to [bend left = 0] (\to);

\end{tikzpicture}
\caption{$\left\|T_1\right\| = 2$ and $\left\|T_2\right\| = 3$.} \label{Cf13}
\end{center}
\end{figure}
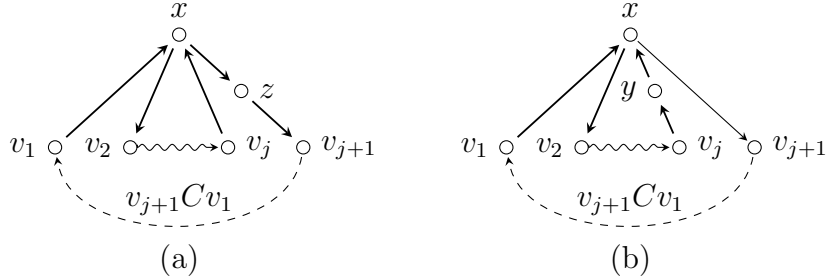

\item $\left\|T_1\right\| = 3$ and $\left\|T_j\right\| = 2$. Let $T_1 = (v_1,x,y,v_2)$ and $T_j = (v_j,z,v_{j+1})$. This case is similar to the previous one.

\item $\left\|T_1\right\| = 3 = \left\|T_j\right\|$. Let $T_1 = (v_1,x,y,v_2)$ and $T_j = (v_j,z,w,v_{j+1})$. 

Consider first that $z = y$ and $w = x$ (Figure \ref{Cf14}). Here, we have that $(v_1,v_{j+1}), (v_j,v_{2}) \in A(H)$. By taking $B_1 = (v_1, v_{j+1}) \cup v_{j+1}Cv_1$ and $B_2 = v_2Cv_j \cup (v_j, v_2)$ we can show that the arcs $(v_1,x), (x,v_{j+1}), (v_j,y)$ and $(y,v_{2})$ are symmetric arcs of $D$ just like we did before, so both arcs $(v_1,v_{2})$ and $(v_j,v_{j+1})$ are symmetric arcs in $C$.

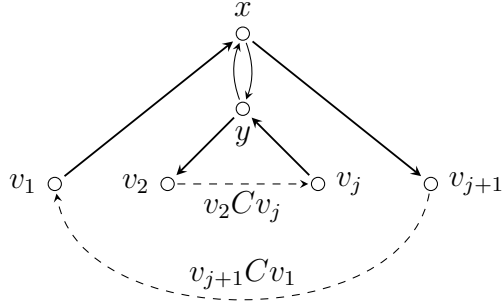
\begin{figure}[ht]
\begin{center}
\begin{tikzpicture}


\node (1) at (-1.5,0)[circle, draw,inner sep=0pt, minimum width=5pt,label=180:$v_1$]{};
\node (2) at (0,0)[circle, draw,inner sep=0pt, minimum width=5pt,label=180:$v_2$]{};
\node (3) at (2,0)[circle, draw,inner sep=0pt, minimum width=5pt,label=0:$v_{j}$]{};
\node (4) at (1,2)[circle, draw,inner sep=0pt, minimum width=5pt,label=90:$x$]{};
\node (5) at (1,1)[circle, draw,inner sep=0pt, minimum width=5pt,label=270:$y$]{};
\node (6) at (3.5,0)[circle, draw,inner sep=0pt, minimum width=5pt,label=0:$v_{j+1}$]{};

\node (C) at (1,-1.2)[]{$v_{j+1}Cv_1$};
\node (P) at (1,-.3)[]{$v_{2}Cv_j$};

\foreach \from/\to in {1/4,5/2,3/5,4/6}
\draw [->, shorten <=2pt, shorten >=2pt, >=stealth, line width=.7pt] (\from) to [bend right = 0] (\to);

\foreach \from/\to in {4/5,5/4}
\draw [->, shorten <=1pt, shorten >=1pt, >=stealth, line width=.4pt] (\from) to [bend left = 20] (\to);

\foreach \from/\to in {6/1}
\draw [->, dashed, shorten <=1pt, shorten >=1pt, >=stealth, line width=.4pt] (\from) to [bend left = 80] (\to);

\foreach \from/\to in {2/3}
\draw [->, dashed, shorten <=1pt, shorten >=1pt, >=stealth, line width=.4pt] (\from) to [bend left = 0] (\to);

\end{tikzpicture}
\caption{The case $z = y$ and $w = x$.} \label{Cf14}
\end{center}
\end{figure}

Now, consider $z \neq y$ and $w = x$ (Figure \ref{Cf15}). In this case, we have $(v_1,v_{j+1}), (v_j,x), (x,v_2) \in A(H)$. By taking $B_1 = (v_1, v_{j+1}) \cup v_{j+1}Cv_1$ we can conclude that $(v_1,x)$ is a symmetric arc in $D$. Now, take $B_2 = (v_j, x, v_2) \cup v_2Cv_j$. It is easy to see that $B_2$ has length less than $n$, so it has a symmetric arc. Again, we can assume that the symmetric arc is either $(v_j,x)$ or $(x,v_2)$. If $(v_j,x)$ is symmetric, there is a directed path from $x$ to $v_j$ of length at most three in $D$. Let $Q$ be such directed path. Here, Lemma \ref{lem1} guarantees that both $(v_j,z)$ and $(z,x)$ are symmetric arcs in $D$, and so $(v_j,v_{j+1})$ is a symmetric arc in $C$. The case where $(x, v_2)$ is symmetric is analogous.

\begin{figure}[ht]
\begin{center}
\begin{tikzpicture}


\node (1) at (-1.5,0)[circle, draw,inner sep=0pt, minimum width=5pt,label=180:$v_1$]{};
\node (2) at (0,0)[circle, draw,inner sep=0pt, minimum width=5pt,label=180:$v_2$]{};
\node (3) at (2,0)[circle, draw,inner sep=0pt, minimum width=5pt,label=0:$v_{j}$]{};
\node (4) at (1,1.5)[circle, draw,inner sep=0pt, minimum width=5pt,label=90:$x$]{};
\node (5) at (.5,.75)[circle, draw,inner sep=0pt, minimum width=5pt,label=180:$y$]{};
\node (6) at (3.5,0)[circle, draw,inner sep=0pt, minimum width=5pt,label=0:$v_{j+1}$]{};
\node (7) at (1.5,.75)[circle, draw,inner sep=0pt, minimum width=5pt,label=0:$z$]{};

\node (C) at (1,-1.2)[]{$v_{j+1}Cv_1$};
\node (P) at (1,-.3)[]{$v_{2}Cv_j$};

\foreach \from/\to in {1/4,5/2,3/7,7/4}
\draw [->, shorten <=2pt, shorten >=2pt, >=stealth, line width=.7pt] (\from) to [bend right = 0] (\to);

\foreach \from/\to in {4/5}
\draw [->, shorten <=1pt, shorten >=1pt, >=stealth, line width=.4pt] (\from) to [bend left = 0] (\to);

\foreach \from/\to in {4/6}
\draw [->, shorten <=1pt, shorten >=1pt, >=stealth, line width=.4pt] (\from) to [bend left = 0] (\to);

\foreach \from/\to in {6/1}
\draw [->, dashed, shorten <=1pt, shorten >=1pt, >=stealth, line width=.4pt] (\from) to [bend left = 80] (\to);

\foreach \from/\to in {2/3}
\draw [->, dashed, shorten <=1pt, shorten >=1pt, >=stealth, line width=.4pt] (\from) to [bend left = 0] (\to);

\end{tikzpicture}
\caption{The case $z \neq y$ and $w = x$.} \label{Cf15}
\end{center}
\end{figure}
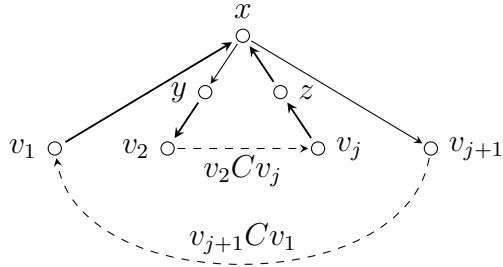

When $z = y$ and $w \neq x$ (see Figure \ref{Cf16}), this case is solved similarly to the case $z \neq y$ and $w = x$.

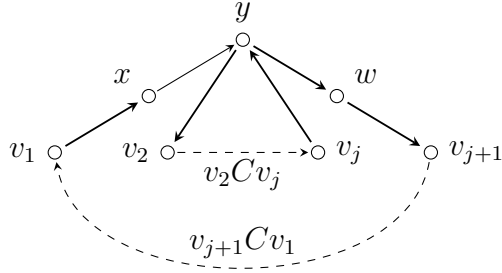
\begin{figure}[ht]
\begin{center}
\begin{tikzpicture}


\node (1) at (-1.5,0)[circle, draw,inner sep=0pt, minimum width=5pt,label=180:$v_1$]{};
\node (2) at (0,0)[circle, draw,inner sep=0pt, minimum width=5pt,label=180:$v_2$]{};
\node (3) at (2,0)[circle, draw,inner sep=0pt, minimum width=5pt,label=0:$v_{j}$]{};
\node (4) at (-.25,.75)[circle, draw,inner sep=0pt, minimum width=5pt,label=160:$x$]{};
\node (5) at (1,1.5)[circle, draw,inner sep=0pt, minimum width=5pt,label=90:$y$]{};
\node (6) at (3.5,0)[circle, draw,inner sep=0pt, minimum width=5pt,label=0:$v_{j+1}$]{};
\node (7) at (2.25,.75)[circle, draw,inner sep=0pt, minimum width=5pt,label=20:$w$]{};

\node (C) at (1,-1.2)[]{$v_{j+1}Cv_1$};
\node (P) at (1,-.3)[]{$v_{2}Cv_j$};

\foreach \from/\to in {1/4,5/2,3/5,5/7,7/6}
\draw [->, shorten <=2pt, shorten >=2pt, >=stealth, line width=.7pt] (\from) to [bend right = 0] (\to);

\foreach \from/\to in {4/5}
\draw [->, shorten <=1pt, shorten >=1pt, >=stealth, line width=.4pt] (\from) to [bend left = 0] (\to);

\foreach \from/\to in {6/1}
\draw [->, dashed, shorten <=1pt, shorten >=1pt, >=stealth, line width=.4pt] (\from) to [bend left = 80] (\to);

\foreach \from/\to in {2/3}
\draw [->, dashed, shorten <=1pt, shorten >=1pt, >=stealth, line width=.4pt] (\from) to [bend left = 0] (\to);

\end{tikzpicture}
\caption{The case $z = y$ and $w \neq x$.} \label{Cf16}
\end{center}
\end{figure}

\end{enumerate}

The three remaining cases are $z = x$ and $w = y$, $z \neq x$ and $w = y$ and finally $z = x$ and $w \neq y$ . This configurations are depicted in Figure \ref{Cf17} (a), (b) and (c), respectively. It is straightforward to check that the result is true when $j = 3$ or $j = n-1$, so we can assume that $3 < j < n-1$.

\begin{figure}[ht]
\begin{center}
\begin{tikzpicture}


\node (1) at (-1,0)[circle, draw,inner sep=0pt, minimum width=5pt,label=180:$v_1$]{};
\node (2) at (0,0)[circle, draw,inner sep=0pt, minimum width=5pt,label=180:$v_2$]{};
\node (3) at (1.3,0)[circle, draw,inner sep=0pt, minimum width=5pt,label=0:$v_{j}$]{};
\node (4) at (.0,1.5)[circle, draw,inner sep=0pt, minimum width=5pt,label=90:$x$]{};
\node (5) at (1.3,1.5)[circle, draw,inner sep=0pt, minimum width=5pt,label=90:$y$]{};
\node (6) at (2.3,0)[circle, draw,inner sep=0pt, minimum width=5pt,label=0:$v_{j+1}$]{};

\node (C) at (.65,-1.5)[]{(a)};
\node (P) at (6.65,-1.5)[]{(b)};
\node (P) at (3.65,-5)[]{(c)};

\node (7) at (5,0)[circle, draw,inner sep=0pt, minimum width=5pt,label=180:$v_1$]{};
\node (8) at (6,0)[circle, draw,inner sep=0pt, minimum width=5pt,label=180:$v_2$]{};
\node (9) at (7.3,0)[circle, draw,inner sep=0pt, minimum width=5pt,label=0:$v_{j}$]{};
\node (10) at (6.65,1.5)[circle, draw,inner sep=0pt, minimum width=5pt,label=90:$y$]{};
\node (11) at (6.975,.75)[circle, draw,inner sep=0pt, minimum width=5pt,label=180:$z$]{};
\node (12) at (8.3,0)[circle, draw,inner sep=0pt, minimum width=5pt,label=0:$v_{j+1}$]{};
\node (13) at (5.825,.75)[circle, draw,inner sep=0pt, minimum width=5pt,label=180:$x$]{};

\node (C) at (.65,-.25)[]{$v_{2}Cv_j$};
\node (P) at (6.65,-.25)[]{$v_{2}Cv_j$};
\node (P) at (3.65,-3.75)[]{$v_{2}Cv_j$};

\node (C) at (.65,-.85+.05)[]{$v_{j+1}Cv_1$};
\node (P) at (6.65,-.85+.05)[]{$v_{j+1}Cv_1$};
\node (P) at (3.65,-4.35+.05)[]{$v_{j+1}Cv_1$};

\node (14) at (2,-3.5)[circle, draw,inner sep=0pt, minimum width=5pt,label=180:$v_1$]{};
\node (15) at (3,-3.5)[circle, draw,inner sep=0pt, minimum width=5pt,label=180:$v_2$]{};
\node (16) at (4.3,-3.5)[circle, draw,inner sep=0pt, minimum width=5pt,label=0:$v_{j}$]{};
\node (20) at (3.65,-2)[circle, draw,inner sep=0pt, minimum width=5pt,label=90:$x$]{};
\node (18) at (4.475,-2.75)[circle, draw,inner sep=0pt, minimum width=5pt,label=0:$w$]{};
\node (19) at (5.3,-3.5)[circle, draw,inner sep=0pt, minimum width=5pt,label=0:$v_{j+1}$]{};
\node (17) at (3.325,-2.75)[circle, draw,inner sep=0pt, minimum width=5pt,label=190:$y$]{};

\foreach \from/\to in {1/4,5/2,3/4,4/5,5/6,7/13,13/10,10/8,9/11,11/10,14/20,16/20,20/18,20/17,17/15,18/19}
\draw [->, shorten <=3pt, shorten >=3pt, >=stealth, line width=.7pt] (\from) to [bend right = 0] (\to);

\foreach \from/\to in {10/12}
\draw [->, shorten <=1pt, shorten >=1pt, >=stealth, line width=.4pt] (\from) to [bend left = 0] (\to);

\foreach \from/\to in {6/1,12/7,19/14}
\draw [->, dashed, shorten <=1pt, shorten >=1pt, >=stealth, line width=.4pt] (\from) to [bend left = 90] (\to);

\foreach \from/\to in {2/3,8/9,15/16}
\draw [->, dashed, shorten <=1pt, shorten >=1pt, >=stealth, line width=.4pt] (\from) to [bend left = 0] (\to);

\end{tikzpicture}
\caption{The three remaining cases.} \label{Cf17}
\end{center}
\end{figure}
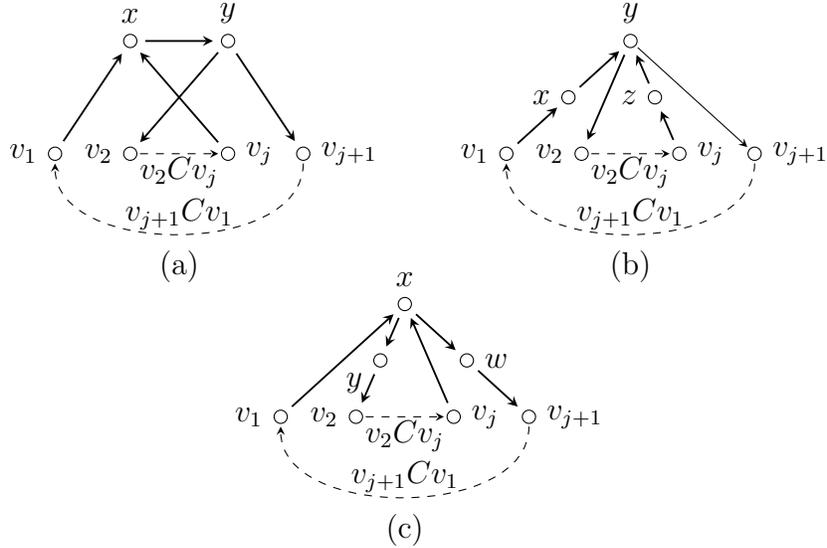

Notice that if $S_1 \cap S_2 \neq \varnothing$ and $S_{j-1} \cap S_j \neq \varnothing$, the $T_1$ and $T_2$ intersect in an $\omega$-configuration, just like $T_{j-1}$ and $T_j$. Here, the minimality of $j-1$ implies that the only possible case is  $w = x$ and $z  \neq y$, which is already covered. This means that $S_1 \cap S_2 \neq \varnothing$ and $S_{j-1} \cap S_j \neq \varnothing$ cannot occur simultaneously. 

Also, if either $S_1 \cap S_2 \neq \varnothing$ or $S_{j-1} \cap S_j \neq \varnothing$, then the case $z = x$ and $w = y$ is excluded due to the minimality of $j-1$.

First, suppose that $S_1 \cap S_2 \neq \varnothing$ and therefore $T_1$ and $T_2$ intersect in an $\omega$-configuration. Let $T_2 = (v_2, y,u,v_3)$. 

If $z = x$ and $w \neq y$, we have that $(v_1, v_{j+1}), (v_j,y), (y,v_3) \in A(H)$. Let $B_1 = (v_1, v_{j+1}) \cup v_{j+1}Cv_1$ and $B_2 = (v_j, y, v_3) \cup v_3Vv_j$. Since $B_1$ and $B_2$ are cycles with length less than $n$, each has a symmetric arc. 

If an arc of $B_2$ other than $(v_j,y)$ or $(y,v_3)$ is symmetric, we are done, thus we can consider that one of them is the symmetric arc. If $(y,v_3)$ is symmetric, then there is a directed path of length at most three from $v_3$ to $y$. Call that directed path $P$. If we take $Q = (y,u,v_3)$, we have that the arcs $(y,u)$ and $(u,v_3)$ are symmetric arcs in $D$ thanks to Lemma \ref{lem1}, so the arc $(v_2,v_3)$ is a symmetric arc in $C$. Hence, assume that $(v_j,y)$ is symmetric. In an analogous way, we get that $(v_j,x)$ and $(x,y)$ are symmetric arcs in $D$.

We can assume that the symmetric arc in $B_1$ is $(v_1,v_{j+1})$, since otherwise we would be done. This means there is a directed path from $v_{j+1}$ to $v_1$ in $D$, so by Lemma \ref{lem2} we have that either $(v_1,x)$ is symmetric or both $(x,w)$ and $(w,v_{j+1})$ are. In the former case, we have $(v_1,v_2)$ is a symmetric arc in $C$. In the later, we get that $(v_j,v_{j+1})$ is the symmetric in $C$.

If $z \neq y$ and $w = x$, we also have that $(v_1, v_{j+1}), (v_j,y), (y,v_3) \in A(H)$. Again, take $B_1 = (v_1, v_{j+1}) \cup v_{j+1}Cv_1$ and $B_2 = (v_j, y, v_3) \cup v_3Cv_j$ and, by the same reasons, assume that the symmetric arc in $B_1$ is $(v_1,v_{j+1})$, and  either $(v_j,y)$ or $(y,v_3)$ is the symmetric arc in $B_2$. 

In this case, Lemma \ref{lem1} guarantees that both $(v_1,x)$ and $(x,v_{j+1})$ are symmetric arcs in $D$. In the same way as in the previous case, if $(y,v_3)$ is symmetric, then $(v_2,v_3)$ is a symmetric arc in $C$. On the other hand, if $(v_j,y)$ is symmetric, Lemma \ref{lem2} gives us that either $(x,y)$ is a symmetric arc in $D$ and thus $(v_1,v_2)$ is a symmetric arc in $C$, or both $(v_j,z)$ and $(z,x)$ are symmetric in $D$, implying that $(v_j,v_{j+1})$ is a symmetric arc in $C$.

Now, suppose that $S_{j-1} \cap S_j \neq \varnothing$ and $T_{j-1}$ and $T_j$ intersect in an $\omega$-configuration. Let $T_{j-1} = (v_2,u,z,v_3)$. It cannot be that $z = x$, because it would contradict the minimality of $j-1$. So it remains to see what happens when $z \neq x$ and $w = y$. In this case, we have that $(v_{j-1}, z), (z,v_2), (v_1,v_{j+1}) \in A(H)$. Take $B_1 = (v_1, v_{j+1}) \cup v_{j+1}Cv_1$ and $B_2 = (v_{j-1}, z, v_2) \cup v_2Cv_{j-1}$. Both $B_1$ and $B_2$ have a symmetric arc, and we can assume that $(v_1,v_{j+1})$ is the symmetric arc in $B_1$ and the one in $B_2$ is either $(v_{j-1}, z)$ or $(z,v_2)$.

If $(v_{j-1}, z)$ is the symmetric arc in $B_2$, then an application of Lemma \ref{lem1} yields that $(v_{j-1},v_{j})$ is a symmetric arc in $C$. If $(z,v_2)$ is the symmetric arc in $B_2$, again Lemma \ref{lem1} gives us that $(z,y)$ and $(y,v_2)$ are symmetric arcs of $D$. 

On the other hand, since $(v_1,v_{j+1})$ is the symmetric arc in $B_1$, applying Lemma \ref{lem2} shows that either $(y,v_{j+1})$ is a symmetric arc in $D$ and thus $(v_j,v_{j+1})$, or both $(v_1,x)$ and $(x,y)$ are symmetric arcs in $D$, hence $(v_1,v_2)$.

For the last part of the proof, we can assume that $S_1 \cap S_2 = \varnothing$ and $S_{j-1} \cap S_j = \varnothing$. The observation about proper $\omega$-blocks that will be the key for the following arguments is this: if one arc in $\iota_{I}$ is symmetric, then there is a symmetric arc in $C$. 

Notice that in the three remaining cases there is a directed path from $v_j$ to $v_2$ of length three contained in $A(T_1) \cup A(T_j)$. Call it $P$. There is as well a path of length three from $v_1$ to $v_{j+1}$ contained in $A(T_1) \cup A(T_j)$, which we will call $Q$. Also there are $i,k$ such that $1 < i < k < j$ and $S_i \cap S_k \neq \varnothing$, then $i + 1 = k$ and there is a proper $\omega$-block $I$ such that $i,k \in I$ and $I \subseteq V(C')$, where $C' = C[\left\{v_2, v_3, \dots, v_{j-1}\right\}]$. Just like before, this means that for every $(v_i, v_{i+1}) \in A(C')$ exactly one of the following conditions is fulfilled:

\begin{itemize}

\item $(v_i, v_{i+1}) \in A(D)$.

\item $S_i \cap S_j = \varnothing$ for every $j \neq i$.

\item There is an $\omega$-block $I$ such that $v_i,v_{i+1} \in I$.

\end{itemize}

Let $\mathfrak{L}, \mathfrak{K}, \Omega, \mathfrak{O}$ and $\alpha$ be defined as follows:

\begin{itemize}

\item $\mathfrak{L} = \left|\left\{T_i: 2 \leq i \leq j, \left\|T_i\right\| = 3\right\}\right|$.

\item $\mathfrak{K} = \left|\left\{T_i: 2 \leq i \leq j, \left\|T_i\right\| = 2\right\}\right|$.

\item $\Omega$ be the set of all the $\omega$-blocks contained in $V(C')$.

\item $\mathfrak{O} = \left|\Omega\right|$.

\item $\alpha =  \left\{(x,y) \in A(C') : \left\{x,y\right\} \not\subset I, I \in \Omega \right\}$.

\end{itemize}

Simple calculations show that the cycle induced by $$\left[A(C')\cap A(D)\right] \cup \left[\bigcup_{(v_i,v_{i+1}) \in \alpha} A(T_i)\right] \cup \left[\bigcup_{I \in \Omega} \iota_I\right] \cup \left[\bigcup_{I \in \Omega} \epsilon_I\right] \cup A(P),$$ which we will call $B_2$, has length $\left\| B_2 \right\| = j + \mathfrak{K} + 2\mathfrak{L} + 2\mathfrak{O} + 1$.

Since $\mathfrak{O} + \mathfrak{L} \leq j-2$ and $\mathfrak{O} + \mathfrak{L} + \mathfrak{K} \leq j-2$, we have the following inequalities:

\begin{center}
\begin{eqnarray*}
\mathfrak{K} + 2\mathfrak{L} + 2\mathfrak{O} &\leq& 2j-4\\
	3\mathfrak{K} + 6\mathfrak{L} + 6\mathfrak{O} &\leq& 2j + 2\mathfrak{K} + 4\mathfrak{L} + 4\mathfrak{O} - 4\\
	\mathfrak{K} + 2\mathfrak{L} + 2\mathfrak{O} &\leq& \frac{2}{3} \left(j + \mathfrak{K} + 2\mathfrak{L} + 2\mathfrak{O} - 2 \right)\\
	\mathfrak{K} + 2\mathfrak{L} + 2\mathfrak{O}  &\leq& \frac{2}{3} \left(j + \mathfrak{K} + 2\mathfrak{L} + 2\mathfrak{O}\right) -\frac{4}{3}\\
	\mathfrak{K} + 2\mathfrak{L} + 2\mathfrak{O} &\leq& \frac{2}{3} \left(j + \mathfrak{K} + 2\mathfrak{L} + 2\mathfrak{O}+ 1 -1 \right) - \frac{4}{3}\\
	\mathfrak{K} + 2\mathfrak{L} + 2\mathfrak{O}  &\leq& \frac{2}{3} \left(j + \mathfrak{K} + 2\mathfrak{L} + 2\mathfrak{O}+ 1 \right) -\frac{2}{3} - \frac{4}{3}\\
	\mathfrak{K} + 2\mathfrak{L} + 2\mathfrak{O} &\leq& \frac{2}{3} \left(j + \mathfrak{K} + 2\mathfrak{L} + 2\mathfrak{O}+ 1 \right) -2
\end{eqnarray*}
\end{center}

By adding $3$ to the right side of the inequality, we get 

\begin{center}
\begin{eqnarray*}
	\mathfrak{K} + 2\mathfrak{L} + 2\mathfrak{O}  &\leq& \frac{2}{3} \left(j + \mathfrak{K} + 2\mathfrak{L} + 2\mathfrak{O} + 1 \right) -2 +3\\
		&\leq& \frac{2}{3} \left(j + \mathfrak{K} + 2\mathfrak{L} + 1 \right) + 1
\end{eqnarray*}
\end{center}

which is the lower bound for the number of symmetric arcs in $D$ of the cycle $B$. If an arc other than the ones in $A(P)$ is symmetric, then $C$ has a symmetric arc. Hence, assume that the three symmetric arcs of $B_2$ are the ones in $A(P)$.

On the other hand, consider the cycle $B_1 = (v_1, v_{j+1}) \cup v_{j+1}Cv_1$. It has length less than $n$, so it has a symmetric arc and we can assume it is $(v_1,v_{j+1})$. Here, an application of Lemma \ref{lem2} yields that at least two arc in $Q$ are symmetric. It is very simple now to verify that the symmetric arc in $P$ plus the two symmetric arcs in $Q$ guarantee the existence of a symmetric arc in $C$.

\end{proof}

Again, the fact that every cycle of $\mathcal{C}^3(D)$ has a symmetric arc implies it is kernel perfect due to Duchet's result. By applying Lemma \ref{k-clo} to the digraph $D$ we get that $D$ has a $4$-kernel, so we have the following result:

\begin{cor}
Let $D$ be a digraph. If every directed cycle $B$ in $D$ has at least $ \frac{2}{3} \left\| B \right\|  + 1$ symmetric arcs, then $D$ has a $4$-kernel.
\end{cor}

\section{Conclusions}

We have proved general sufficient conditions for the existence of $3$- and $4$-kernels in digraphs.   Moreover,
these are the first such results since the generalization of Richardson's Theorem due to Kwa\'snik in 1981.    Also,
if true, Conjecture \ref{Conj-k} would give general sufficient conditions for the existence of $k$-kernels in general,
and it would be a natural generalization of Duchet's Theorem (Theorem \ref{Duchet}), again, a result from 1980.

By analyzing the proof of Theorem \ref{4k}, it becomes clear that we need to find a new strategy if we want to prove the proposed conjecture. It is not hard to see that several $\omega$-like configurations will emerge if the same technique is used to study other particular cases of the conjecture, as well as the general case. 

However, provided the conjecture is true, it may become a very useful tool in the study of $k$-kernels. A possible line of work, side by side with the conjecture, is to find (whenever possible) generalizations to $k$-kernels of the results that make use of Duchet's Theorem.





\end{document}